\documentclass{amsart}
\usepackage{amsmath}
\usepackage{amsthm,amscd,amssymb,amsfonts, amsbsy}
\usepackage{latexsym}
\usepackage{pxfonts, mathrsfs}

\usepackage{color}
\usepackage{enumerate}

\numberwithin{equation}{section}

\theoremstyle{plain}
\newtheorem{theorem}[equation]{Theorem}
\newtheorem{lemma}[equation]{Lemma}

\theoremstyle{definition}
\newtheorem{definition}[equation]{Definition}

\newtheorem{remark}[equation]{Remark}
\newtheorem*{acknowledgment}{Acknowledgment}

\theoremstyle{remark}

\newcommand{\dv}{\operatorname{div}}

\newcommand{\dist}{\operatorname{dist}}
\newcommand{\diam}{\operatorname{diam}}

\newcommand{\curl}{\operatorname{curl}}

\renewcommand{\vec}[1]{\boldsymbol{#1}}

\newcommand{\bR}{\mathbb R}

\newcommand{\bZ}{\mathbb Z}

\newcommand\cD{\mathcal{D}}

\newcommand\cL{\mathscr{L}}

\newcommand\cS{\mathcal{S}}

\newcommand\cLt{{}^t\!\mathscr{L}}

\newcommand{\VMO}{\mathrm{VMO}}

\newcommand{\ip}[1]{\left\langle#1\right\rangle}

\providecommand{\set}[1]{\{#1\}}

\providecommand{\bigset}[1]{\bigl\{#1\bigr\}}

\providecommand{\abs}[1]{\lvert#1\rvert}
\providecommand{\Abs}[1]{\left\lvert#1\right\rvert}

\providecommand{\norm}[1]{\lVert#1\rVert}

\renewcommand{\epsilon}{\varepsilon}
\renewcommand{\qedsymbol}{$\blacksquare$}

\begin{document}
\title[Elliptic systems with measurable coefficients]
{Elliptic systems with measurable coefficients of the type of Lam\'e system in three dimensions.}

\author[K. Kang]{Kyungkeun Kang}
\address[K. Kang]{Department of Mathematics, Yonsei University, Seoul 120-749, Republic of Korea}
\email{kkang@yonsei.ac.kr}

\author[S. Kim]{Seick Kim}
\address[S. Kim]{Department of Mathematics, Yonsei University, Seoul 120-749, Republic of Korea}
\email{kimseick@yonsei.ac.kr}



\begin{abstract}
We study the $3 \times 3$ elliptic systems $\nabla \times (a(x) \nabla\times u)-\nabla (b(x) \nabla \cdot u)=f$, where the coefficients $a(x)$ and $b(x)$ are positive scalar functions that are measurable and bounded away from zero and infinity.
We prove that weak solutions of the above system are H\"older continuous under some minimal conditions on the inhomogeneous term $f$.
We also present some applications and discuss several related topics including estimates of the Green's functions and the heat kernels of the above systems.
\end{abstract}

\maketitle

\section{Introduction}
\label{sec:intro}
In this article, we are concerned with the system of equations
\begin{equation}					\label{eq1.5ee}
\nabla\times (a(x)\nabla\times \vec u)-\nabla(b(x)\nabla \cdot \vec u)= \vec f \quad\text{in }\;\Omega,
\end{equation}
where the unknown $\vec u=(u^1,u^2, u^3)$ and the inhomogeneous term $\vec f=(f^1,f^2,f^3)$ are  vector valued functions defined on a (possibly unbounded) domain $\Omega\subseteq \bR^3$, and the coefficients $a(x)$ and $b(x)$ are positive scalar functions on $\Omega$ that are measurable and bounded away from zero and infinity.
It should be noted from the beginning that the above system \eqref{eq1.5ee} is elliptic.
As a matter of fact, the following vector identity
\begin{equation}					\label{eq0.2ab}
\nabla \times (\nabla \times \vec u)- \nabla (\nabla \cdot \vec u)= -\Delta \vec u
\end{equation}
implies that in the case when $a$ and $b$ are constants, the above system reduces to
\[
-a\Delta \vec u+(a-b) \nabla (\nabla \cdot \vec u)=\vec f\quad\text{in }\;\Omega,
\]
which (under the assumption that $a>0$ and $b>4a/3$) becomes the Lam\'e system of linearized elastostatics in dimension three; see e.g., Dahlberg et al. \cite{DKV}.
A special case of the system \eqref{eq1.5ee} is the following system
\begin{equation}				\label{eq0.1aa}
\nabla\times (a(x)\nabla\times \vec u)=0,\quad  \nabla \cdot \vec u=0\quad\text{in }\;\Omega,
\end{equation}
which arises from Maxwell's equations in a quasi-static electromagnetic field, where the displacement of the electric current is neglected; see e.g., Landau et al. \cite[Ch.~VII]{LLP}.
In \cite{KK02}, the authors proved that  weak solutions of the system \eqref{eq0.1aa} are H\"older continuous in $\Omega$; see also Yin \cite{Yin02}.
It is an interesting result because in general, weak solutions of elliptic systems with bounded measurable coefficients in dimension three or higher are not necessarily continuous; see De Giorgi \cite{DG68}.
Another motivation for studying the system \eqref{eq1.5ee} comes from an interesting article by Giaquinta and Hong \cite{GH}, where they considered the following equations involving differential forms:
\begin{equation}					\label{eq0.2hg}
d^* (\sigma(x) d A)=0,\quad -d^* A=0\quad\text{in }\;\Omega,
\end{equation}
where $\sigma(x) \in L^\infty(\Omega)$ is a function with $\sigma_1 \leq \sigma(x) \leq \sigma_2\,$, $\sigma_1$ and $\sigma_2$ being two positive constants, $A$ is a one-form, $dA$ is its exterior differential, and $d^*$ denotes the adjoint of $d$ (i.e., $d^*=\delta$, the codifferential).
Related to the well-known result of De Giorgi \cite{DG57} on elliptic equations, they raised an interesting question of whether any weak solution $A$ of the equations \eqref{eq0.2hg} is H\"older continuous in $\Omega$.
In the three dimensional setting, the equations \eqref{eq0.2hg} becomes the system \eqref{eq0.1aa}, and thus, in dimension three, a positive answer was given in \cite{KK02}.
Conversely, in terms of differential forms, the system \eqref{eq1.5ee} with $\vec f=0$ becomes
\[
d^* (a(x) dA)+d(b(x) d^* A)=0\quad\text{in }\;\Omega;\quad A=u^1 dx^1+ u^2 dx^2+ u^3 dx^3.
\]
Similar to the question raised by Giaquinta and Hong \cite{GH}, it is natural to ask whether weak solutions of the above equations are H\"older continuous in $\Omega$.
We hereby thank Marius Mitrea for suggesting this question to us.

In this article, we prove that weak solutions of the system \eqref{eq1.5ee} are H\"older continuous in $\Omega$ assuming a minimal condition on $\vec f$, and thus give a positive answer to the above question in dimension three; see Theorem~\ref{thm3.2a} below for the precise statement.
With this H\"older estimate at hand, we are able to show that there exists a unique Green's function $\vec G(x,y)$ of the system \eqref{eq1.5ee} in an arbitrary domain $\Omega\subseteq \bR^3$, and it has the natural bound
\[
\abs{\vec G(x,y)} \leq N \abs{x-y}^{-1}
\]
for all $x, y\in\Omega$ such that $0<\abs{x-y}<d_x \wedge d_y$, where $d_x:=\dist(x,\partial\Omega)$, $a\wedge b:=\min(a,b)$, and $N$ is a constant independent of $\Omega$.
In particular, when $\Omega=\bR^3$, the above estimate holds for all $x \neq y$; see Theorem~\ref{thm5.6gr} and Remark~\ref{rmk6.7gr} below.
It also follows that the heat kernel $\vec K_t(x,y)$ of the system \eqref{eq1.5ee} exists in any domain $\Omega$, and in the case when $\Omega=\bR^3$, we have the following usual Gaussian bound for $\vec K_t(x,y)$; see Theorem~\ref{thm2hk} below:
\[
\abs{\vec K_t(x,y)} \leq N t^{-3/2}\exp\{-\kappa|x-y|^2/ t \},\quad \forall t>0,\;\; x,y\in\bR^3.
\]
Another goal of this article is to establish a global H\"older estimate for weak solutions of the system \eqref{eq0.1aa} in bounded Lipschitz domains.
More precisely, we consider the following Dirichlet problem
\begin{equation}					\label{eq0.3cc}
\left\{
\begin{array}{c}
\nabla\times (a(x)\nabla\times \vec u)=\vec f +\nabla\times \vec g \quad\text{in }\;\Omega,\\
\nabla \cdot \vec u=h\quad\text{in }\;\Omega,\\
\vec u=0\quad \text{on }\;\partial\Omega,
\end{array}
\right.
\end{equation}
where $\Omega$ is a bounded, simply connected Lipschitz domain.
We prove that the weak solution $\vec u$ of the above problem \eqref{eq0.3cc} is uniformly H\"older continuous in $\overline\Omega$ under some suitable conditions on the inhomogeneous terms $\vec f$, $\vec g$, and $h$; see Theorem~\ref{thm3.1t} for the details.
This question of global H\"older regularity for weak solutions of the system \eqref{eq0.1aa} turned out to be a rather delicate problem and was not discussed at all in \cite{KK02}.
Yin addressed this issue in \cite{Yin02}, but it appears that there is a serious flaw in his proof;  he also considered a similar problem with a more general boundary condition in \cite{Yin04}, but it seems to us that his argument there regarding estimate near the boundary has a gap too.
Utilizing the above mentioned global H\"older estimate for weak solutions of the system \eqref{eq0.3cc}, we show that the Green's function $\vec G(x,y)$ of the system \eqref{eq0.1aa} in $\Omega$ has the following global bound:
\[
\abs{\vec G(x,y)}  \leq N \bigset{d_x\wedge \abs{x-y}}^{\alpha} \bigset{d_y\wedge \abs{x-y}}^{\alpha} \abs{x-y}^{-1-2\alpha},\quad \forall x, y\in \Omega,\;\; x\neq y,
\]
where $0<\alpha<1$; see Theorem~\ref{thm5.8gr} for the details.
In that case, we also have the following global estimate for the heat kernel $\vec K_t(x,y)$ of the system \eqref{eq0.1aa} in $\Omega$:
For all $T>0$, there exists a constant $N$ such that for all  $x,y \in \Omega$ and $0<t \leq T$, we have
\[
\abs{\vec K_t(x,y)} \leq N \left(1 \wedge \frac {d_x} {\sqrt {t} \vee \abs{x-y}} \right)^{\alpha} \left(1 \wedge \frac{d_y} {\sqrt {t} \vee \abs{x-y}}\right)^{\alpha}\,t^{-3/2}\exp \set{-\kappa \abs{x-y}^2/t},
\]
where $\kappa>0$ and $\alpha\in (0,1)$ are constants independent of $T$, and we used the notation $a\vee b=\max(a,b)$; see Theorem~\ref{thm3hk} below.
At the moment, it is not clear to us whether or not any global H\"older estimate is available for weak solutions of the full system \eqref{eq1.5ee} with zero Dirichlet boundary data.

The organization of the paper is as follows. In Section~\ref{sec:nd}, we introduce some related notation and definitions.
In Section~\ref{sec:main}, we state our main theorems and give a few remarks concerning extensions of them.
The proofs of our main results are given in Section~\ref{sec:pf} and some applications of them are presented in Section~\ref{sec:app}.
We devote Section~\ref{sec:green} entirely to the study of the Green's functions of the system \eqref{eq1.5ee}, and  Section~\ref{sec:p} to the investigation of the parabolic system and the heat kernels associated to the system \eqref{eq1.5ee}.

\section{Notation and Definitions}			\label{sec:nd}
\subsection{Basic notation}
The basic notation used in this article are those employed in Gilbarg and Trudinger \cite{GT}.
A Function in bold symbol such as $\vec u$ means that it is a three dimensional vector-valued function; $\nabla\cdot \vec u$ denotes $\dv \vec u$, $\nabla \times \vec u$ denotes $\curl \vec u$, and $\nabla \vec u$ denotes the gradient matrix of $\vec u$.
Throughout the article, $\Omega$ denotes a (possibly unbounded) domain in $\bR^3$ (i.e., an open connected set in $\bR^3$) and $\partial \Omega$ denotes its boundary.
For a domain $\Omega$ with $C^1$ boundary $\partial\Omega$, we denote by $\vec n$ the unit outward normal to $\partial\Omega$.
Let $L$ be the operator of the form
\[
L\vec u:=\nabla\times (a(x)\nabla\times \vec u)-\nabla(b(x)\nabla \cdot \vec u)
\]
whose coefficients are measurable functions on $\Omega$ satisfying the following condition:
\begin{equation}					\label{eq0.2bh}
\nu \leq a(x),\; b(x) \leq \nu^{-1},\quad\forall x\in\Omega,\quad \text{for some }\;\nu\in (0,1].
\end{equation}
For $x\in\Omega$ and $r>0$, we denote $B_r(x)$ the open ball of radius $r$ centered at $x$ and
\[
\Omega_r(x):= \Omega\cap B_r(x); \quad (\partial\Omega)_r(x):=\partial\Omega\cap B_r(x).
\]
We write $S' \subset\subset S$ if $S'$ has a compact closure in $S$; $S'$ is strictly contained in $S$.
\subsection{Function spaces}
The H\"older spaces $C^{k,\alpha}(\overline \Omega)$ ($C^{k,\alpha}(\Omega)$) are defined as the subspaces of $C^k(\overline \Omega)$ ($C^k(\Omega)$) consisting of functions whose $k$-th order partial derivatives are uniformly H\"older continuous (locally H\"older continuous) with exponent $\alpha$ in $\Omega$.
For simplicity we write
\[
C^{0,\alpha}(\Omega)=C^\alpha(\Omega),\quad C^{0,\alpha}(\overline\Omega)=C^\alpha(\overline \Omega),
\]
with the understanding $0<\alpha<1$ whenever this notation is used.
We set
\[
\norm{u}_{C^\alpha(\overline\Omega)}=\abs{u}_{0,\alpha;\Omega}=[u]_{\alpha;\Omega}+\abs{u}_{0;\Omega} := \sup_{\substack{x, y \in \Omega\\ x\neq y}} \frac{\abs{u(x)-u(y)}}{\abs{x-y}^\alpha}+\sup_{\Omega}\,\abs{u}.
\]
For $p\geq 1$, we let $L^p(\Omega)$ denote the classical Banach space consisting of measurable functions on $\Omega$ that are $p$-integrable.
The norm in $L^p(\Omega)$ is defined by
\[
\norm{u}_{p; \Omega}=\norm{u}_{L^p(\Omega)}=\left(\int_\Omega \abs{u}^p\,dx\right)^{1/p}.
\]
For $p \geq 1$ and $k$ a non-negative integer, we let $W^{k,p}(\Omega)$ the usual Sobolev space; i.e.
\[
W^{k,p}(\Omega)=\set{ u\in L^p(\Omega): D^\alpha u \in L^p(\Omega)\;\,\text{for all}\;\, \abs{\alpha} \leq k}.
\]
We denote by $C^\infty_0(\Omega)$ the set of all functions in $C^\infty(\Omega)$ with compact support in $\Omega$.
Some other notations are borrowed from Galdi \cite{Galdi} and Mal\'y and Ziemer \cite{MZ}.
Setting
\[
\cD=\cD(\Omega)=\bigset{\vec u \in C^\infty_0(\Omega): \nabla\cdot \vec u=0\;\text{ in }\;\Omega},
\]
for $q\in[1,\infty)$ we denote by $H_q(\Omega)$ the completion of $\cD(\Omega)$ in the norm of $L^q$.
The space $Y^{1,2}(\Omega)$ is defined as the family of all weakly differentiable functions $u\in L^{6}(\Omega)$, whose weak derivatives are functions in $L^2(\Omega)$.
The space $Y^{1,2}(\Omega)$ is endowed with the norm
\[
\norm{u}_{Y^{1,2}(\Omega)}:=\norm{u}_{L^{6}(\Omega)}+\norm{\nabla u}_{L^2(\Omega)}.
\]
If $\abs{\Omega}<\infty$, then H\"older's inequality implies that $Y^{1,2}(\Omega)\subset W^{1,2}(\Omega)$.
We define $Y^{1,2}_0(\Omega)$ as the closure of $C^\infty_0(\Omega)$ in $Y^{1,2}(\Omega)$.
In the case $\Omega = \bR^3$, we have  $Y^{1,2}(\bR^3)=Y^{1,2}_0(\bR^3)$.
Notice that by the Sobolev inequality, it follows that
\begin{equation}					\label{eqP-14}
\norm{u}_{L^{6}(\Omega)} \leq N \norm{\nabla u}_{L^2(\Omega)},\quad \forall u\in Y^{1,2}_0(\Omega).
\end{equation}
Therefore, we have $W^{1,2}_0(\Omega)\subset Y^{1,2}_0(\Omega)$ and $W^{1,2}_0(\Omega)=Y^{1,2}_0(\Omega)$ if $\abs{\Omega}<\infty$; see \cite[\S 1.3.4]{MZ}.
In particular, if $\Omega$ is a bounded domain, then we have $Y^{1,2}_0(\Omega)=W^{1,2}_0(\Omega)$.
\subsection{Lipschitz domain}
We say that $\Omega\subset \bR^3$ is a (bounded) Lipschitz domain  if 
\begin{enumerate}[i)]
\item
$\Omega$ is a bounded domain; i.e.
\[
\diam \Omega:=\sup\set{\abs{x-y}: x, y\in \Omega} <\infty,
\]
\item
There are constants $M$ and $r_0>0$, called Lipschitz character of $\partial\Omega$, such that for each $P\in \partial\Omega$, there exists a rigid transformation of coordinates such that  $P=0$ and
\[
\Omega \cap B_{r_0}=\set{x=(x',x_3)\in \bR^3: x_3> \varphi (x')}\cap B_{r_0};\quad B_{r_0}=B_{r_0}(0),
\]
where $\varphi:\bR^2 \to \bR$ is a Lipschitz function such that $\varphi(0)=0$, with Lipschitz constant less than or equal to $M$; i.e.
\[
\abs{\varphi(x')-\varphi(y')} \leq M \abs{x'-y'},\quad \forall x',y'\in \bR^2.
\]
\end{enumerate}

\subsection{Weak solutions}					\label{sec2.2ws}
We say that $\vec u$ is a weak solution in $Y^{1,2}(\Omega)$ of the system \eqref{eq1.5ee} if
\begin{equation}					\label{eq2.1ws}
\int_\Omega a (\nabla \times \vec u) \cdot (\nabla \times \vec \phi) + b (\nabla \cdot \vec u) (\nabla \cdot \vec \phi) = \int_\Omega \vec f \cdot \vec \phi,\quad \forall \vec \phi \in C^\infty_0(\Omega).
\end{equation}
We say that a function $\vec u$ is a weak solution in $Y^{1,2}_0(\Omega)$ of the problem
\begin{equation}					\label{eq2.3ws}
\left\{
\begin{array}{c}
\nabla\times (a(x)\nabla\times \vec u)-\nabla(b(x)\nabla \cdot \vec u)= \vec f \quad\text{in }\;\Omega,\\
\vec u=0\quad \text{on }\;\partial\Omega,
\end{array}
\right.
\end{equation}
if $\vec u$ belongs to $Y^{1,2}_0(\Omega)$ and satisfies the identity \eqref{eq2.1ws}.
By a weak solution in $Y^{1,2}_0(\Omega)$ of the problem \eqref{eq0.3cc}, we mean a function $\vec u\in Y^{1,2}_0(\Omega)$ satisfying
\begin{align}					\label{eq2.4ws}
\int_\Omega a (\nabla \times \vec u) \cdot (\nabla \times \vec \phi) &= \int_\Omega \vec f \cdot \vec \phi + \vec g \cdot (\nabla \times \vec \phi) ,\quad \forall \vec \phi \in C^\infty_0(\Omega)\\
							\label{eq2.5ws}
\int_\Omega \vec u \cdot \nabla \psi&=-\int_\Omega h \psi,\quad \forall \psi\in C^\infty_0(\Omega).
\end{align}
By using the standard elliptic theory, one can easily prove the existence and uniqueness of a weak solution of the problem \eqref{eq2.3ws} in $Y^{1,2}_0(\Omega)$ provided $\vec f \in L^{6/5}(\Omega)$.
Similarly, if $\vec f\in H_{6/5} (\Omega)$ and $\vec g\in L^2(\Omega)$, one can show that there exists a weak solution in $Y^{1.2}_0(\Omega)$ of the problem \eqref{eq0.3cc}  when $h=0$;  in the more general case when $h \in L^{6/5}(\Omega)$ and $\int_\Omega h = 0$, one can show that there exists a unique weak solution in $Y^{1.2}_0(\Omega)$ of the problem \eqref{eq0.3cc} provided that $\Omega$ is a bounded Lipschitz domain; see Appendix for the proofs.

\section{Main Results}			\label{sec:main}
Our first theorem says that if $\vec f \in L^{q}(\Omega)$ with $q>3/2$, then weak solutions of the system \eqref{eq1.5ee} are locally H\"older continuous in $\Omega$. 

\begin{theorem}			\label{thm3.2a}
Let $\Omega$ be a (possibly unbounded) domain in $\bR^3$.
Assume that $a(x)$ and $b(x)$ are measurable functions on $\Omega$ satisfying the condition \eqref{eq0.2bh}, and that $\vec u\in Y^{1,2}(\Omega)$ is a weak solution of the system \eqref{eq1.5ee}, where $\vec f \in L^q(\Omega)$ with $q>3/2$.
Then $\vec u$ is H\"older continuous in $\Omega$, and for all $B_R=B_R(x_0) \subset\subset \Omega$, we have the following estimate for $\vec u$:
\begin{equation}			\label{eq3.3cc}
R^\alpha [\vec u]_{\alpha; B_{R/2}} + \abs{\vec u}_{0;B_{R/2}} \leq N \left( R^{-3/2} \norm{\vec u}_{L^2(B_R)}+ R^{2-3/q} \norm{\vec f}_{L^q(B_R)}\right),
\end{equation}
where $\alpha=\alpha(\nu, q) \in (0,1)$ and $N=N(\nu,q)>0$.
\end{theorem}

In order to establish a global H\"older estimate for weak solutions of the problem \eqref{eq0.3cc}, we need to impose some conditions on $\Omega$.
We shall assume that $\Omega$ is a bounded Lipschitz domain whose first homology group $H_1(\Omega;\bR)$ is trivial; i.e.,
\begin{equation}					\label{eq2.5ef}
H_1(\Omega; \bR)=0.
\end{equation}
For example, if $\Omega$ is simply connected, then it satisfies the above condition.
As mentioned in \S\ref{sec2.2ws}, the existence and uniqueness of a weak solution in $W^{1,2}_0(\Omega)$ of the problem \eqref{eq0.3cc} is established by a standard argument; see Appendix.
\begin{theorem}			\label{thm3.1t}
Let $\Omega \subset \bR^3$ be a bounded Lipschitz domain satisfying the condition \eqref{eq2.5ef}.
Let $a(x)$ be a measurable function on $\Omega$ satisfying the condition \eqref{eq0.2bh} and $\vec u\in W^{1,2}_0(\Omega)$ be the weak solution of the problem \eqref{eq0.3cc}, where $\vec f\in H_{q/2}(\Omega)$, $\vec g, h \in L^q(\Omega)$ for some $q>3$, and $\int_\Omega h=0$.
Then, $\vec u$ is  uniformly H\"older continuous in $\Omega$ and satisfies the following estimate:
\begin{equation}			\label{eq3.2bx}
\norm{\vec u}_{C^{\alpha}(\overline \Omega)}\le N \left(\norm{\vec f}_{L^{q/2}(\Omega)}+\norm{\vec g}_{L^q(\Omega)} +\norm{h}_{L^q(\Omega)}  \right),
\end{equation}
where $\alpha=\alpha(\nu, q, \Omega) \in (0,1)$ and $N=N(\nu,q,\Omega)>0$.
\end{theorem}

Related to the above theorems, several remarks are in order.
\begin{remark}					\label{rmk2.5rr}
In Theorem~\ref{thm3.2a}, one may assume that $a(x)$ is not a scalar function but a $3\times 3$ symmetric matrix valued function satisfying
\[
\nu \abs{\vec \xi}^2 \leq \vec \xi^T a(x)\vec \xi \leq \nu^{-1} \abs{\vec \xi}^2,\quad \forall \vec \xi \in \bR^3,\;\;\forall x\in\Omega,\;\; \text{for some}\;\nu\in (0,1].
\]
There is no essential change in the proof; see \cite{KKM}.
As a matter of fact, one may drop the symmetry assumption on $a(x)$ if one assume further that $a\in L^\infty(\Omega)$.
\end{remark}

\begin{remark}
In Theorem~\ref{thm3.2a}, instead of assuming $\vec f \in L^q(\Omega)$, one may assume that $\vec f$ belongs to the Morrey space $L^{p,\lambda}$ with $p=6/5$ and $\lambda=6(1+2\delta)/5$ for some $\delta\in(0,1)$; see the proof of Theorem~\ref{thm4.2b} and Remark~\ref{rmk4.5ff} in Section~\ref{sec:p}.
The ``interior'' Morrey space $L^{p,\lambda}$ is defined to be the set of all functions $f\in L^p(\Omega)$ with finite norm
\[
\norm{u}_{L^{p,\lambda}}=\sup_{B_r(x_0)\subset \Omega} \left(r^{-\lambda}\int_{B_r(x_0)} \abs{u}^p\,\right)^{1/p}.
\]
Moreover, instead of the system \eqref{eq1.5ee}, one may consider the following system:
\[
\nabla\times (a(x)\nabla\times \vec u)-\nabla(b(x)\nabla \cdot \vec u) = \vec f + \nabla \times \vec F + \nabla g \quad \text{in }\;\Omega.
\]
One can show that weak solutions $\vec u$ of the above system are H\"older continuous in $\Omega$ if 
\[
\vec f \in L^{6/5,6(1+2\delta)/5},\;\; \vec F\in L^{2,(1+2\delta)/2},\;\;\text{and }\; g\in L^{2,(1+2\delta)/2}\quad \text{for some }\;\delta \in(0,1).
\]
In particular, if $\vec f \in L^{q/2}(\Omega)$, $\vec F \in L^q(\Omega)$, and $g\in L^q(\Omega)$ for $q>3$, then weak solutions $\vec u\in Y^{1,2}(\Omega)$ of the above system are H\"older continuous in $\Omega$.
Moreover, in that case, we have the estimate
\[
r^\alpha [\vec u]_{\alpha; B_{r/2}} + \abs{\vec u}_{0; B_{r/2}} \leq N \left( r^{-3/2} \norm{\vec u}_{2;B_r}+ r^{2-6/q} \norm{\vec f}_{q/2;B_r}+ r^{1-3/q} \norm{\vec F}_{q; B_r}+r^{1-3/q}  \norm{g}_{q;B_r}\right),
\]
whenever $B_r=B_r(x_0)\subset\subset \Omega$, where $\alpha=\alpha(\nu, q)\in (0,1)$ and $N=N(\nu, q)$.
\end{remark}

\begin{remark}					\label{rmk2.10}
In Theorem~\ref{thm3.1t}, one may wish to consider the following problem with non-zero Dirichlet boundary data, instead of the problem \eqref{eq0.3cc}:
\begin{equation}					\label{eq2.13cc}
\left\{
\begin{array}{c}
\nabla\times (a(x)\nabla\times \vec u)=\vec f +\nabla\times \vec g \quad\text{in }\;\Omega,\\
\nabla \cdot \vec u=h\quad\text{in }\;\Omega,\\
\vec u= \vec \psi \quad \text{on }\;\partial\Omega,
\end{array}
\right.
\end{equation}
where one needs to assume the compatibility condition $\int_\Omega h = \int_{\partial\Omega} \vec \psi\cdot n$ instead of the condition $\int_\Omega h =0$ in Theorem~\ref{thm3.1t}.
If $\vec \psi$ is the trace of a Sobolev function $\vec w \in W^{1,q}(\Omega)$ with $q>3$,  then $\vec v:=\vec u-\vec w$ is a solution of the problem \eqref{eq0.3cc} with $\vec g$ and $h$ replaced respectively by  $\tilde{\vec g}$ and $\tilde h$, where
\[
\tilde{\vec g}:=\vec g - a \nabla\times \vec w,\quad \tilde h:=h-\nabla \cdot \vec w \in L^q(\Omega).
\]
Notice that $\int_\Omega \tilde h =0$.
Therefore, by the estimate \eqref{eq3.2bx} and Morrey's inequality, we have the following estimate  the weak solution $\vec u$ of the problem \eqref{eq2.13cc}:
\[
\norm{\vec u}_{C^{\alpha}(\overline \Omega)} \leq  N \left(\norm{\vec f}_{L^{q/2}(\Omega)}+\norm{\vec g}_{L^q(\Omega)} +\norm{h}_{L^q(\Omega)} + \norm{\vec w}_{W^{1,q}(\Omega)} \right),
\]
where $\alpha=\alpha(\nu, q, \Omega) \in (0,1)$ and $N=N(\nu,q,\Omega)>0$.
Recall that $\Omega\subset \bR^3$ is assumed to be a bounded Lipschitz domain.
It is known that if $\vec \psi$ belongs to the Besov space $B^q_{1-1/q}(\partial\Omega)$, then it can be extended to a function $\vec w$ in the Sobolev space $W^{1,q}(\Omega)$ in such a way that the following estimate holds:
\[
\norm{\vec w}_{W^{1,q}(\Omega)} \leq N \norm{\vec \psi}_{B^q_{1-1/q}(\partial\Omega)},
\]
where $N=N(\Omega,q)$; see e.g., Jerison and Kenig \cite[Theorem~3.1]{JK95}.
Therefore, the following estimate is available for the weak solution $\vec u$ of the problem \eqref{eq2.13cc}:
\[
\norm{\vec u}_{C^{\alpha}(\overline \Omega)}\le N \left(\norm{\vec f}_{L^{q/2}(\Omega)}+\norm{\vec g}_{L^q(\Omega)} +\norm{h}_{L^q(\Omega)} + \norm{\vec \psi}_{B^q_{1-1/q}(\partial\Omega)} \right),
\]
where $\alpha=\alpha(\nu, q, \Omega) \in (0,1)$ and $N=N(\nu,q,\Omega)>0$.
The above estimate provides, in particular, the global bounds for the weak solution $\vec u$ of the problem \eqref{eq2.13cc} in $\Omega$.
It seems to us that Theorem~\ref{thm3.1t} is the first result establishing the global boundedness of weak solutions of the Dirichlet problem \eqref{eq2.13cc} in Lipschitz domains.
\end{remark}

\section{Proofs of  Main theorems}			\label{sec:pf}

\subsection{Proof of Theorem~\ref{thm3.2a}}
We shall make the qualitative assumption that the weak solution $\vec u$ is smooth in $\Omega$.
This can be achieved by assuming coefficient $a(x)$ and the inhomogeneous term $\vec f$ are smooth in $\Omega$ and adopting the standard approximation argument.
It should be clear from the proof that the constant $\alpha$ and $N$ will not depend on these extra smoothness assumption.
By a standard computation (see e.g., \cite[Lemma~4.4]{KK02}), we can derive the following Caccioppoli's inequality for $\vec u$:
\begin{lemma}[Caccioppoli's inequality]			\label{lem4.2tt}
With $\vec u$, $\vec f$, and $R$ as in the theorem, we have
\[
\int_{B_{2r}} \abs{\nabla \times \vec u}^2+ \abs{\nabla \cdot \vec u}^2 \leq N \left(r^{-2} \int_{B_{3r}} \abs{\vec u}^2+ \norm{\vec f }_{L^{6/5}(B_{3r})}^2\right);\quad r=R/3.
\]
\end{lemma}

We take the divergence in the system \eqref{eq1.5ee} to get
\[
-\Delta \psi= \nabla \cdot \vec f \quad\text{in }\;\Omega;\quad \psi:=b \nabla \cdot \vec u.
\]
Denote $B(x)= 1/b(x)$ and observe that $\vec u$ satisfies
\begin{equation}					\label{eq4.27bb}
\nabla\cdot \vec u  = B \psi \quad \text{in }\;\Omega.
\end{equation}
Next, we split $\vec u=\vec v+\vec w$ in $B_r=B_r(x_0)$, where $r=R/3$ and $\vec v$ is a solution of the problem
\[
\left\{
\begin{array}{c}
\nabla \cdot \vec v = B \psi- (B \psi)_{x_0,r}  \quad \text{in }\;B_r,\\
\vec v= 0 \quad \text{on }\;\partial B_r,
\end{array}
\right.
\]
where we used the notation \[(B \psi)_{x_0,r}:=\fint_{B_r(x_0)} B \psi.\]
We assume that the function $\vec v$ is chosen so that following estimate, which is originally due to Bogovski\v{i} \cite{Bog}, holds for $\vec v$ (see Galdi \cite[\S III.3]{Galdi}):
\begin{equation}					\label{eq4.30sh}
\norm{\nabla \vec v}_{L^p(B_r)} \leq N \norm{B \psi-(B \psi)_{x_0,r}}_{L^p(B_r)} \leq N \norm{B \psi}_{L^p(B_r)},\;\;\forall p \in (1,\infty);\quad N=N(p).
\end{equation}
We decompose $\psi=\psi_1+\psi_2$ in $B_{2r}$, where $\psi_2$ is the solution of
\[
\left\{
\begin{array}{c}
- \Delta \psi_2 = \nabla \cdot \vec f \quad \text{in }\;B_{2r},\\
\psi_2= 0 \quad \text{on }\;\partial B_{2r}.
\end{array}
\right.
\]
By the Calder\`on-Zygmund theory, we have
\[
\norm{\nabla \psi_2}_{L^q(B_{2r})} \le N \norm{\vec f}_{L^q(B_{2r})}.
\]
We assume without loss of generality that $q<3$.
Then by the Sobolev inequality, we have
\[
\norm{\psi_2}_{L^{q*}(B_{2r})} \le N \norm{\vec f}_{L^q(B_{2r})},\quad q*=\tfrac{3q}{3-q}>3.
\]
Note that $\psi_1$ is harmonic in $B_{2r}$ and thus by the mean value property, we have
\[
\norm{\psi_1}_{L^q(B_r)}\le N r^{3/q-3/2} \norm{\psi_1}_{L^2(B_{2r})}.
\]
Therefore, by using $\psi_1= \psi-\psi_2$ and H\"older's inequality, we obtain
\begin{equation}					\label{eq4.31se}
\norm{\psi}_{L^q(B_r)}\le N r^{3/q-3/2} \norm{\psi}_{L^2(B_{2r})}+ N r \norm{\vec f}_{L^{q}(B_{2r})}.
\end{equation}
Combining the estimates \eqref{eq4.30sh} and \eqref{eq4.31se}, and then using \eqref{eq4.27bb} and Lemma~\ref{lem4.2tt}, we get
\begin{equation}					\label{eq4.29sk}
\norm{\nabla \vec v}_{L^q(B_r )} \leq N r^{3/q-5/2} \norm{\vec u}_{L^2(B_{3r})}+N r \norm{\vec f}_{L^{q}(B_{3r})} ;\quad N=N(\nu,q).
\end{equation}
By Sobolev inequality, \eqref{eq4.30sh}, \eqref{eq4.27bb}, Lemma~\ref{lem4.2tt}, and H\"older's inequality, we also estimate
\begin{equation}					\label{eq4.33rr}
\norm{\vec v}_{L^2(B_r)} \leq N r \norm{\nabla \vec v}_{L^2(B_r)} 
\leq N \left( \norm{\vec u}_{L^2(B_{3r})}+ r^{7/2-3/q}  \norm{\vec f}_{L^{q}(B_{3r})} \right).
\end{equation}

On the other hand, note that $\vec w=\vec u-\vec v$ is a weak solution of the problem
\[
\left\{
\begin{array}{c}
\nabla\times (a(x)\nabla\times \vec w)=  \vec f+ \nabla \psi - \nabla\times (a(x)\nabla\times \vec v) \quad\text{in }\; B_r,\\
\nabla \cdot \vec w= (B \psi)_{x_0,r} \quad\text{in }\;B_r.
\end{array}
\right.
\]
We remark that in the proof of  \cite[Theorem~2.1]{KK02}, we used the condition $\nabla\cdot \vec u=0$ only to establish the following equality (recall the identity \eqref{eq0.2ab} above),
\[\nabla\times (\nabla \times \vec u)=-\Delta \vec u,\]
which can be also obtained by merely assuming that $\nabla \cdot \vec u$ is constant.
Therefore, by \cite[Theorem~2.1 and Remark~2.10]{KK02}, we have (via a standard scaling argument)
\begin{equation}						\label{eq4.35mm}
r^\alpha[\vec w]_{\alpha; B_{r/2}} \leq N \left(r^{-3/2} \norm{\vec w}_{L^2(B_r)} +  r^{2-3/q}\norm{\vec f +\nabla \psi}_{L^{q}(B_r)}+  r^{2-3/q}\norm{\nabla \vec v}_{L^{q*}(B_r)} \right),
\end{equation}
where $\alpha=\alpha(\nu, q)>0$, $N=N(\nu, q)$, and we used $\nabla \cdot (\vec f + \nabla \psi)=0$ and $q*=\frac{3q}{3-q}>3$.

We estimate the RHS of \eqref{eq4.35mm} as follows.
By the estimate \eqref{eq4.33rr}, we have
\begin{equation}
\norm{\vec w}_{L^2(B_r)} \leq \norm{\vec u}_{L^2(B_r)}+\norm{\vec v}_{L^2(B_r)} \leq  N \left( \norm{\vec u}_{L^2(B_{3r})} + r^{7/2-3/q}  \norm{\vec f}_{L^{q}(B_{3r})}\right).
\end{equation}
Recall that $\psi=\psi_1+\psi_2$ and that $\psi_2$ satisfies
\[
\norm{\psi_2}_{L^2(B_{2r})} \le  \norm{\psi_2}_{L^{q*}(B_{2r})}\; \abs{B_{2r}}^{\frac{1}{2}-\frac{1}{q*}} \le N r^{5/2-3/q} \norm{\vec f}_{L^q(B_{2r})}.
\]
Since $\psi_1=\psi-\psi_2$ is harmonic in $B_{2r}$, we also have
\[
\norm{\nabla \psi_1}_{L^{q}(B_r)} \le N r^{3/q-5/2}\norm{\psi-\psi_2}_{L^2(B_{2r})}
 \le N r^{3/q-7/2} \norm{\vec u}_{L^2(B_{3r})} + N \norm{\vec f}_{L^q(B_{3r})}.
\]
where we used \eqref{eq4.27bb}, Lemma~\ref{lem4.2tt}, and H\"older's inequality as well as the previous inequality.
Therefore, we obtain
\begin{equation}						\label{eq4.37ch}
\norm{\nabla \psi}_{L^{q}(B_r)} \le
\norm{\nabla \psi_1}_{L^{q}(B_r)} + \norm{\nabla \psi_2}_{L^{q}(B_r)} \le
N r^{3/q-7/2} \norm{\vec u}_{L^2(B_{3r})} + N \norm{\vec f}_{L^q(B_{3r})}.
\end{equation}
Similar to \eqref{eq4.29sk}, we have
\[
\norm{\nabla \vec v}_{L^{q*}(B_r )} \le N r^{3/q-5/2} \norm{\psi}_{L^2(B_{2r})}+N \norm{\vec f}_{L^{q}(B_{2r})} \le
N r^{3/q-7/2} \norm{\vec u}_{L^2(B_{3r})}+N \norm{\vec f}_{L^{q}(B_{3r})}.
\]
By combining \eqref{eq4.35mm} -- \eqref{eq4.37ch}, and the above inequality, we obtain
\[
r^\alpha [\vec w]_{\alpha; B_{r/2}} \leq N \left( r^{-3/2} \norm{\vec u}_{L^2(B_{3r})} + r^{2-3/q}  \norm{\vec f}_{L^{q}(B_{3r})}\right).
\]
Also, by Morrey's inequality, we have
\[
[\vec v]_{\mu; B_r} \le N \norm{\nabla v}_{L^{q*}(B_r)} \le  N \left( r^{-3/2-\mu} \norm{\vec u}_{L^2(B_{3r})} +  \norm{\vec f}_{L^{q}(B_{3r})}\right);\quad \mu=2-3/q.
\]
By combining the above two estimates (replace $\alpha$ by $\mu$ if necessary), we conclude
\begin{equation}					\label{eq4.40zz}
r^{\alpha} [\vec u]_{\alpha; B_{r/2}} \leq N  \left( r^{-3/2} \norm{\vec u}_{L^2(B_{3r})} + r^{2-3/q}  \norm{\vec f}_{L^{q}(B_{3r})}\right).
\end{equation}
From the above estimate \eqref{eq4.40zz}, we can estimate $\abs{\vec u}_{0;B_{r/4}}$ as follows.
For all $y\in B_{r/4}$, the triangle inequality yields
\[
\abs{\vec u(y)}\leq \abs{\vec u(x)}+ [\vec u]_{\alpha; B_{r/2}} (r/2)^\alpha, \quad \forall x\in B_{r/4}.
\]
Taking the average over $B_{r/4}$ in $x$, and then using H\"older's inequality and \eqref{eq4.40zz}, we get
\[
\abs{\vec u(y)} \leq \left(\fint_{B_{r/4}} \abs{\vec u}^2\right)^{1/2}+N \left( r^{-3/2} \norm{\vec u}_{L^2(B_{3r})} + r^{2-3/q}  \norm{\vec f}_{L^{q}(B_{3r})}\right).
\]
Since the above estimate is uniform in $y\in B_{r/4}$, we thus have
\begin{equation}					\label{eq4.41yx}
\abs{\vec u}_{0; B_{r/4}} \leq N \left( r^{-3/2} \norm{\vec u}_{L^2(B_{3r})} + r^{2-3/q}  \norm{\vec f}_{L^{q}(B_{3r})}\right).
\end{equation}
Recall that $r=R/3$. 
Therefore, the desired estimate \eqref{eq3.3cc} follows from \eqref{eq4.40zz} and \eqref{eq4.41yx} and the standard covering argument.
The theorem is proved.
\hfill\qedsymbol

\subsection{Proof of Theorem~\ref{thm3.1t}}
We shall again make the qualitative assumption that  the coefficient $a(x)$, the inhomogeneous terms $\vec f$, $\vec g$, $h$, and the domain $\Omega$ are smooth.
By a standard elliptic regularity theory, we may then assume that $\vec u$ is also smooth in $\overline\Omega$.
In this proof, we denote by $N$ a constant that depends only on $\nu$, $q$, and $\Omega$, unless explicitly otherwise stated.
It should be emphasized that those constants $N$ employed in various estimates below, do not inherit any information from the extra smoothness assumption imposed on $\Omega$; its dependence on $\Omega$ will be only that on the Lipschitz character $M, r_0$ of $\partial\Omega$ and $\diam\Omega$.
Let us recall the following lemma, the proof of which can be found in \cite{KK05}.

\begin{lemma}				\label{lem:1}
Let $\vec f \in \cD(\Omega)$, where $\Omega$ is a domain in $\bR^3$.
Then, there exists $\vec F\in C^\infty(\Omega)$ such that $\nabla\times \vec F=\vec f$ in $\Omega$.
Moreover, for any $p\in (1,\infty)$, we have
\[
\norm{\nabla \vec F}_{L^p(\Omega)}\le N \norm{\vec f}_{L^p(\Omega)};\quad N=N(p).
\]
\end{lemma}

By using the above lemma, we may write $\vec f=\nabla \times \vec F$, where $\vec F\in C^\infty(\Omega)$ satisfies the following estimate
\begin{equation}				\label{eq4.1aa}
\norm{\nabla \vec F}_{L^{q/2}(\Omega)}\le N \norm{\vec f}_{L^{q/2}(\Omega)};\quad N=N(q).
\end{equation}
Notice that $\vec u$ then satisfies
\begin{equation}					\label{eq4.3cc}
\nabla\times(a(x) \nabla\times \vec u -\vec F -\vec g) =0 \quad \text{in }\;\Omega.
\end{equation}
Let $\varphi$ be a solution of the Neumann problem
\begin{equation}					\label{eq4.4dd}
\left\{
\begin{array}{c}
\Delta \varphi = \nabla\cdot (a(x) \nabla\times \vec u -\vec F -\vec g) \quad \text{in }\;\Omega,\\
\partial \varphi/\partial n=-(\vec F+\vec g)\cdot \vec n \quad \text{on }\;\partial\Omega,
\end{array}
\right.
\end{equation}
where $\vec n$ denotes the outward unit normal vector of $\partial\Omega$.
Recall that $\varphi$ is unique up to an additive constant.
We shall hereafter fix $\varphi$ by assuming $\fint_\Omega \varphi=0$.

\begin{lemma}					\label{lem:1-1}
With $\vec u$ and $\varphi$ given as above, we have
\begin{equation}					\label{eq4.5ee}
\nabla \varphi =  a(x) \nabla\times \vec u -\vec F -\vec g \quad \text{in }\;\Omega.
\end{equation}
\end{lemma}

\begin{proof}
First we claim that the boundary condition $\vec u=0$ on $\partial\Omega$ implies that
\begin{equation}				\label{eq4.6cj}
(\nabla \times \vec u) \cdot \vec n =0\quad \text{on }\;\partial\Omega.
\end{equation}
To see this, take any surface $\cS \subset \partial\Omega$ with a smooth boundary $\partial \cS \subset \partial\Omega$.
By Stokes' theorem, we then have
\[
\iint_\cS (\nabla \times \vec u)\cdot \vec n\, dS = \int_{\partial \cS} \vec u \cdot d\vec r=0.
\]
Since $\cS$ is arbitrary and $(\nabla \times \vec u)\cdot \vec n$ is continuous, we have $(\nabla \times \vec u) \cdot \vec n =0$ on $\partial\Omega$ as claimed.
Next, we set 
\[
\vec G=\nabla \varphi -  a(x) \nabla\times \vec u +\vec F +\vec g.
\]
The lemma will follow if we prove that $\vec G \equiv 0$ in $\Omega$.
By \eqref{eq4.3cc} we have $\nabla\times \vec G =0$ in $\Omega$, and thus by the condition \eqref{eq2.5ef}, there exists a potential $\psi$ such that $\vec G=\nabla \psi$ in $\Omega$.
Then by \eqref{eq4.4dd} and \eqref{eq4.6cj}, we find that $\psi$ satisfies $\Delta \psi=0$ in $\Omega$ and $\partial\psi/\partial n =0$ on $\partial\Omega$.
Therefore, we must have $\vec G= \nabla \psi =0$ in $\Omega$.
The lemma is proved.
\end{proof}

Hereafter, we shall denote $A(x)= 1/a(x)$.
It follows from \eqref{eq0.2bh} that
\[
\nu \leq A(x) \leq \nu^{-1},\quad\forall x\in\Omega.
\]
Observe that from \eqref{eq4.5ee} we have 
\[
0=\nabla\cdot (\nabla \times \vec u) = \nabla\cdot \bigl[A(x) \bigr(\nabla \varphi +\vec F +\vec g\bigr)\bigr],
\]
and thus by Lemma~\ref{lem:1-1} we find that $\varphi$ satisfies the following conormal problem:
\begin{equation}					\label{eq4.7aa}
\left\{
\begin{array}{c}
\dv (A(x) \nabla \varphi) = -\dv \left(A \vec F + A \vec g \right) \quad \text{in }\;\Omega,\\
(A(x)\nabla \varphi) \cdot \vec n=-(A\vec F+A\vec g)\cdot \vec n \quad \text{on }\;\partial\Omega.
\end{array}
\right.
\end{equation}
In the variational formulation, \eqref{eq4.7aa} means that we have the identity
\begin{equation}					\label{eq4.10bv}
\int_\Omega A\nabla \varphi \cdot \nabla \zeta = -\int_\Omega (A \vec F + A \vec g)\cdot \nabla \zeta,\quad \forall \zeta\in W^{1,2}(\Omega).
\end{equation}
In particular, by using $\varphi$ itself as a test function, we get
\[
\norm{\nabla \varphi}_{L^2(\Omega)} \leq N \left( \norm{\vec F}_{L^2(\Omega)}+ \norm{\vec g}_{L^2(\Omega)}\right);\quad N=N(\nu).
\]
By Poincar\'e's inequality (recall $\fint_\Omega \varphi=0$) and H\"older's inequality, we then have 
\[
\norm{\varphi}_{W^{1,2}(\Omega)} \leq N \left( \norm{\vec F}_{L^q(\Omega)}+ \norm{\vec g}_{L^q(\Omega)}\right).
\]
Moreover, one can obtain the following estimate by utilizing \eqref{eq4.10bv} and adjusting, for example, the proof of \cite[Theorem~8.29]{GT} (see \cite{LU} and also \cite[\S VI.10]{Lieberman}):
\begin{equation}					\label{eq4.11ew}
[\varphi]_{\mu; \Omega}\leq N\left( \norm{\vec F}_{L^{q}(\Omega)}+\norm{\vec g}_{L^q(\Omega)}\right);\quad \mu=\mu(\nu,q,\Omega)\in (0,1).
\end{equation}

Then, by Campanato's integral characterization of H\"older continuous functions (see e.g., \cite[Theorem~1.2, p. 70]{Gi83}), we derive from \eqref{eq4.11ew} that
\begin{equation}					\label{eq4.16kt}
\int_{\Omega_r(x_0)} \Abs{\varphi-\varphi_{x_0,r}}^2 \leq N r^{3+2\mu} \left(\norm{\vec F}_{L^{q}(\Omega)}+\norm{\vec g}_{L^q(\Omega)}\right)^2; \quad \varphi_{x_0,r}:= \fint_{\Omega_r(x_0)} \varphi.
\end{equation}
From the identity \eqref{eq4.10bv}, we also obtain the following Caccioppoli's inequality:
\begin{equation}					\label{eq4.17hh}
\int_{\Omega_{r/2}(x_0)} \Abs{\nabla \varphi}^2 \leq  N r^{-2}\int_{\Omega_r(x_0)} \Abs{\varphi-\varphi_{x_0,r}}^2+ N r^{3-6/q} \left(\norm{\vec F}_{L^q(\Omega)}^2 + \norm{\vec g}_{L^q(\Omega)}^2\right).
\end{equation}
Setting $\gamma=\min(\mu, 1-3/q)$, and combining \eqref{eq4.16kt} and \eqref{eq4.17hh}, we get the following Morrey-Campanato type estimate for $\nabla \varphi$:
\begin{equation}					\label{eq4.9rr}
\int_{\Omega_r(x_0)} \abs{\nabla \varphi}^2 \leq N r^{1+2\gamma} \left( \norm{\vec F}_{L^{q}(\Omega)}+\norm{\vec g}_{L^q(\Omega)}\right)^2, \quad \forall x_0\in\Omega,\;\; \forall r \in (0, \diam\Omega).
\end{equation}

Having the estimate \eqref{eq4.9rr}  together with the boundary condition $\vec u=0$ on $\partial\Omega$, which is assumed to be locally Lipschitz, we now derive a global H\"older estimate for $\vec u$ as follows.
Since $\nabla\cdot \vec u=h$, by \eqref{eq0.2ab} and \eqref{eq4.5ee} we see that $\vec u$ satisfies
\[
-\Delta \vec u=\nabla\times(A \nabla\varphi)+ \nabla\times(A \vec F + A \vec g)-\nabla h\quad\text{in }\;\Omega.
\]
By H\"older's inequality, we find that (recall $\gamma \leq 1-3/q$)
\[
\int_{\Omega_r(x_0)} \abs{\vec F+ \vec g}^2 \leq  N r^{1+2\gamma}  \left( \norm{\vec F}_{L^{q}(\Omega)}+\norm{\vec g}_{L^q(\Omega)}\right)^2,\quad \forall x_0\in\Omega,\;\; \forall r \in (0, \diam\Omega),
\]
where we used the assumption that $\diam \Omega<\infty$.
Similarly, H\"older's inequality yields
\[
\int_{\Omega_r(x_0)} \abs{h}^2 \leq  N r^{1+2\gamma}\norm{h}_{L^{q}(\Omega)}^2,\quad \forall x_0\in\Omega,\;\; \forall r \in (0, \diam\Omega).
\]
Setting $\vec G:=A(\nabla \varphi+ \vec F +\vec g)$, we find that $\vec u$ satisfies
\begin{equation}					\label{eq4.17de}
\left\{
\begin{array}{c}
-\Delta \vec u = \nabla\times \vec G -\nabla h \quad \text{in }\;\Omega,\\
\vec u=0 \quad \text{on }\;\partial\Omega,
\end{array}
\right.
\end{equation}
where $\vec G$ and $h$ satisfies the following estimate for all $x_0\in\Omega$ and $0< r <\diam\Omega$:
\begin{equation}					\label{eq4.18kk}
\int_{\Omega_r(x_0)} \abs{\vec G}^2 +\abs{h}^2 \leq N r^{1+2\gamma} \left( \norm{\vec F}_{L^{q}(\Omega)}+\norm{\vec g}_{L^q(\Omega)}+\norm{h}_{L^q(\Omega)}\right)^2.
\end{equation}

Observe that the identity \eqref{eq4.5ee} implies that $\nabla \times \vec u$ enjoys the  Morrey-Campanato type estimate \eqref{eq4.9rr}.
The following lemma asserts that in fact, the ``full gradient'' $\nabla \vec u$ satisfies a similar estimate.
\begin{lemma}
With $\vec u$ given as above, there exists $\alpha=\alpha(\nu,q,\Omega)\in(0,1)$ such that
for all $x_0\in\Omega$ and $0<r < \diam\Omega$, we have
\begin{equation}					\label{eq4.20ys}
\int_{\Omega_r(x_0)} \abs{\nabla \vec u}^2 \leq N r^{1+2\alpha} \left( \norm{\vec F}_{L^{q}(\Omega)}+\norm{\vec g}_{L^q(\Omega)}+\norm{h}_{L^q(\Omega)}\right)^2.
\end{equation}
\end{lemma}
\begin{proof}
We decompose $\vec u=\vec v+\vec w$ in $\Omega_r(x_0)$, where $\vec v$ is the solution of
\[
\left\{
\begin{array}{c}
-\Delta \vec v = 0 \quad \text{in }\;\Omega_r(x_0),\\
\vec v=\vec u \quad \text{on }\;\partial\Omega_r(x_0).
\end{array}
\right.
\]
Notice that each $v^i$ ($i=1,2,3$) is a harmonic function vanishing on $(\partial\Omega)_r(x_0)\subset \partial\Omega$.
By a well-known boundary H\"older regularity theory for harmonic functions in Lipschitz domains, there exists $\beta=\beta(\Omega)\in (0,1)$ and $N=N(\Omega)$ such that
\begin{equation}					\label{eq4.21np}
\int_{\Omega_\rho(x_0)} \abs{\nabla \vec v}^2 \leq  N\left(\frac{\rho}{r}\right)^{1+2\beta} \int_{\Omega_r(x_0)} \abs{\nabla \vec v}^2,\quad \forall \rho \in (0,r].
\end{equation}
On the other hand, observe that $\vec w=\vec u-\vec v$ is a weak solution of the problem
\[
\left\{
\begin{array}{c}
-\Delta \vec w = \nabla\times\vec G - \nabla h \quad \text{in }\;\Omega_r(x_0),\\
\vec w= 0 \quad \text{on }\;\partial\Omega_r(x_0).
\end{array}
\right.
\]
By using $\vec w$ itself as a test function in the above equations and utilizing \eqref{eq4.18kk}, we derive
\begin{equation}				\label{eq4.22xm}
\int_{\Omega_r(x_0)} \abs{\nabla \vec w}^2 \leq N \int_{\Omega_r(x_0)} \abs{\vec G}^2+\abs{h}^2 \leq N r^{1+2\gamma} \left( \norm{\vec F}_{L^{q}(\Omega)}+\norm{\vec g}_{L^q(\Omega)}+\norm{h}_{L^q(\Omega)}\right)^2.
\end{equation}
By combining \eqref{eq4.21np} and \eqref{eq4.22xm}, we get for any $\rho\leq r$,
\[
\int_{\Omega_\rho(x_0)} \abs{\nabla \vec u}^2 \leq N\left(\frac{\rho}{r}\right)^{1+2\beta} \int_{\Omega_r(x_0)} \abs{\nabla \vec u}^2 + N r^{1+2\gamma} \left( \norm{\vec F}_{L^{q}(\Omega)}+\norm{\vec g}_{L^q(\Omega)}+\norm{h}_{L^q(\Omega)}\right)^2.
\]
Take any $\alpha>0$ such that $\alpha < \min(\beta,\gamma)$ and applying a well-known iteration argument (see e.g. \cite[Lemma~2.1, p. 86]{Gi83}), for all $x_0\in\Omega$ and $0<r <R \leq \diam\Omega$, we have
\[
\int_{\Omega_r(x_0)} \abs{\nabla \vec u}^2 \leq N\left(\frac{r}{R}\right)^{1+2\alpha} \int_{\Omega_R(x_0)} \abs{\nabla \vec u}^2 + N r^{1+2\alpha} \left( \norm{\vec F}_{L^{q}(\Omega)}+\norm{\vec g}_{L^q(\Omega)}+\norm{h}_{L^q(\Omega)}\right)^2.
\]
The lemma follows from the above estimate (take $R=\diam\Omega$) and the estimate
\begin{equation}				\label{eq4.23nx}
\int_{\Omega} \abs{\nabla \vec u}^2 \leq N \int_{\Omega} \abs{\vec G}^2 + \abs{\nabla h}^2 \leq N\left( \norm{\vec F}_{L^{q}(\Omega)}+\norm{\vec g}_{L^q(\Omega)}+\norm{h}_{L^q(\Omega)}\right)^2,
\end{equation}
which is obtained by using $\vec u$ itself as a test function in \eqref{eq4.17de} and then applying \eqref{eq4.18kk} with $r=\diam \Omega$.
The lemma is proved.
\end{proof}

We now estimate $[\vec u]_{\alpha;\Omega}$ as follows.
Denote by $\tilde{\vec u}$ the extension of $\vec u$ by zero on $\bR^3\setminus \Omega$.
Notice that $\tilde{\vec u} \in W^{1,2}(\bR^3)$ and $\nabla \tilde{\vec u}=\chi_{\Omega} \nabla \vec u$.
Then by Poincar\'e's inequality and \eqref{eq4.20ys}, we find that for all $x\in\Omega$ and $0<r<\diam\Omega$, we have
\[
\int_{B_r(x)} \Abs{\tilde{\vec u}-\tilde{\vec u}_{x,r}}^2\leq  N r^{3+2\alpha}\left( \norm{\vec F}_{L^{q}(\Omega)}+\norm{\vec g}_{L^q(\Omega)}+\norm{h}_{L^q(\Omega)}\right)^2.
\]
By a standard argument in the boundary regularity theory, it is readily seen that the above estimate is valid for all $x\in B_R(x_0)$ and $r<2R$, where $x_0\in\Omega$ and $R=\diam\Omega$.
 Therefore, by the Campanato's integral characterization of H\"older continuous functions, we find that $\tilde{\vec u}$ is uniformly H\"older continuous in $B_R(x_0)\supset \overline \Omega$ with the estimate
 \begin{equation}				\label{eq4.22vm}
 [\tilde{\vec u}]_{\alpha; B_R(x_0)} \leq N \left(\norm{\vec F}_{L^{q}(\Omega)}+\norm{\vec g}_{L^q(\Omega)}+\norm{h}_{L^q(\Omega)}\right).
 \end{equation}
 The above estimate \eqref{eq4.22vm} clearly implies that
 \begin{equation}					\label{eq4.23qt}
[\vec u]_{\alpha; \Omega}\le N\left( \norm{\vec F}_{L^{q}(\Omega)}+\norm{\vec g}_{L^q(\Omega)}+\norm{h}_{L^q(\Omega)}\right).
\end{equation}
Finally, we estimate of $\abs{\vec u}_{0; \Omega}$ similar to \eqref{eq4.41yx}.
For $x_0\in \Omega$, the triangle inequality yields
\[
\abs{\vec u(x_0)}\leq \abs{\vec u(x)}+ [\tilde{\vec u}]_{C^{0,\alpha}(\overline B_R(x_0))} R^\alpha, \quad \forall x\in\Omega;\quad R=\diam\Omega.
\]
Taking the average over $\Omega$ in $x$, and then using H\"older's inequality and \eqref{eq4.22vm}, we have
\[
\abs{\vec u(x_0)} \leq \left(\fint_\Omega \abs{\vec u}^2\right)^{1/2}+N(\diam\Omega)^\alpha \left(\norm{\vec F}_{L^{q}(\Omega)}+\norm{\vec g}_{L^q(\Omega)}+\norm{h}_{L^q(\Omega)}\right).
\]
On the other hand, by \eqref{eq4.23nx} and the Poincar\'e's inequality, we have
\[
\int_\Omega \abs{\vec u}^2 \leq N \int_\Omega \abs{\nabla \vec u}^2 \leq N\left( \norm{\vec F}_{L^{q}(\Omega)}+\norm{\vec g}_{L^q(\Omega)}+\norm{h}_{L^q(\Omega)}\right)^2.
\]
Therefore, by combining the above two inequalities, we obtain
 \begin{equation}					\label{eq4.26dj}
\abs{\vec u}_{0;\Omega}\leq N\left( \norm{\vec F}_{L^{q}(\Omega)}+\norm{\vec g}_{L^q(\Omega)}+\norm{h}_{L^q(\Omega)}\right).
\end{equation}
The desired estimate \eqref{eq3.2bx} now follows from \eqref{eq4.23qt}, \eqref{eq4.26dj}, \eqref{eq4.1aa}, and the Sobolev's inequality.
The proof is complete.
\hfill\qedsymbol

\section{Applications}				\label{sec:app}
\subsection{Quasilinear system}
As a first application, we consider the quasilinear system,
\begin{equation}				\label{eq5.1an}
\nabla\times (\mathcal A(x,\vec u)\nabla\times \vec u)-\nabla(\mathcal B(x,\vec u)\nabla \cdot \vec u)= \vec f \quad\text{in }\;\Omega.
\end{equation}
Here we assume $\mathcal A, \mathcal B:\Omega\times\bR^3\to\bR$ satisfy the following conditions:
\begin{enumerate}[i)]
\item
$\nu\leq \mathcal A, \mathcal B \leq \nu^{-1}$ for some constants $\nu \in (0,1]$.
\item
$\mathcal A$ and $\mathcal B$ are H\"older continuous in $\Omega\times\bR^3$; i.e. $\mathcal A, \mathcal B \in C^\mu(\Omega\times\bR^3)$ for $\mu\in (0,1)$.
\end{enumerate}

\begin{theorem}				\label{thm:app1}
Let $\mathcal A$ and $\mathcal B$ satisfy the above conditions and let $\vec u\in Y^{1,2}(\Omega)$ be a weak solution of the system \eqref{eq5.1an} with $\vec f \in L^q(\Omega)$ for $q>3$.
Then, we have $\vec u \in C^{1,\alpha}(\Omega)$, where $\alpha=\min(\mu, 1-3/q)$.
\end{theorem}
\begin{proof}
By Theorem~\ref{thm3.2a}, we know $\vec u \in C^{\beta}(\Omega)$ for some $\beta \in (0,1)$.
Then the coefficients $a(x):=\mathcal A(x,\vec u(x))$ and $b(x):=\mathcal B(x,\vec u(x))$ are H\"older continuous with some exponent $\gamma \in (0,1)$.
The rest of proof relies on the well-known ``freezing coefficients'' method in Schauder theory and is omitted; c.f. \cite[Theorem~2.2]{KK02}.
\end{proof}

\begin{remark}
In Theorem~\ref{thm:app1}, if one assumes instead that $\mathcal A, \mathcal B \in C^{k,\mu}(\Omega\times \bR^3)$ and $\vec f \in C^{k-1,\mu}(\Omega)$ with $k \in \bZ_{+}$ and $\mu\in(0,1)$, then one can show that $\vec u\in C^{k+1,\mu}(\Omega)$; in particular, $\vec u$ becomes a classical solution of the system \eqref{eq5.1an}.
\end{remark}

\subsection{Maxwell's system in quasi-static electromagnetic fields with temperature effect}
As mentioned in the introduction, the problem \eqref{eq0.3cc} arises from the Maxwell's system in a quasi-static electromagnetic field.
Especially, if the electric conductivity strongly depends on the temperature, then by taking the temperature effect into consideration the classical Maxwell system in a quasi-static electromagnetic field reduces to the following mathematical model (see Yin \cite{Yin97}):
\[
\left\{\begin{array}{c}
\vec H_t+\nabla \times(\rho(u)\nabla \times \vec H)=0,\\
\nabla\cdot \vec H=0,\\
u_t-\Delta u=\rho(u)\, \abs{\nabla\times \vec H}^2,
\end{array}\right.
\]
where $\vec H$ and $u$ represents, respectively, the strength of the magnetic field and temperature while $\rho(u)$ denotes the electrical resistivity of the material, which is assumed to be bounded below and above by some positive constants; i.e.,
\begin{equation}					\label{eq6.3ap}
\nu \leq \rho \leq v^{-1}\;\text{ for some }\;\nu\in (0,1].
\end{equation}
We are thus lead to consider the following Dirichlet problem in the steady-state case:
\begin{equation}			\label{eq6.4ii}
\left\{\begin{array}{c}
\nabla\times(\rho(u)\nabla\times \vec H)=0\quad\text{in }\;\Omega,\\
\nabla\cdot \vec H=0\quad\text{in }\;\Omega,\\
\vec H = \vec \Psi\quad\text{on }\;\partial\Omega,\\
-\Delta u=\rho(u)\,\abs{\nabla\times \vec H}^2\quad\text{in }\;\Omega,\\
u = \phi\quad\text{on }\;\partial\Omega,
\end{array}\right.
\end{equation}
where we assume that $\vec \Psi$ and $\phi$ are functions in $W^{1,q}(\Omega)$ for $q>3$.
Existence of a pair of weak solutions $(\vec H, u)$ was proved in Yin \cite{Yin97} and local H\"older continuity of the pair $(\vec H, u)$ in $\Omega$ was proved by the authors in \cite{KK02}.
Here, we prove that the pair $(\vec H, u)$ is indeed uniformly H\"older continuous in $\overline \Omega$.
\begin{theorem}
\label{thm:2-1}
Let $\Omega$ satisfy the hypothesis of Theorem~\ref{thm3.1t} and $\rho$ satisfy the condition \eqref{eq6.3ap}.
Let $(\vec H,u)$ be the weak solution of  the problem \eqref{eq6.4ii}.
Then we have $(\vec H,u)\in C^\alpha(\overline \Omega)$ for some $\alpha \in (0,1)$.
In particular, $\vec H$ and  $u$ are bounded in $\Omega$.
\end{theorem}

\begin{proof}
By Theorem~\ref{thm3.1t} and Remark~\ref{rmk2.10},  we find that $\vec H \in C^\alpha(\overline \Omega)$ for some $\alpha \in (0,1)$ and satisfies the estimate
\begin{equation}					\label{eq6.9nn}
\norm{\vec H}_{C^{\alpha}(\overline \Omega)}\le N  \norm{\vec \Psi}_{W^{1,q}(\Omega)}.
\end{equation}
Also, notice from \eqref{eq4.20ys} and Remark~\ref{rmk2.10} that for all $x_0\in\Omega$ and $0<r < \diam\Omega$, we have
\begin{equation}					\label{eq6.10qv}
\int_{\Omega_r(x_0)} \abs{\nabla \vec H}^2 \leq N r^{1+2\alpha} \norm{\vec \Psi}_{W^{1,q}(\Omega)}^2.
\end{equation}
On the other hand, using the vector calculus identity,
\[
\nabla \cdot(\vec F\times \vec G)=(\nabla \times \vec F)\cdot \vec G-\vec F\cdot(\nabla \times \vec G),
\]
together with the first equation $\nabla \times(\rho(u)\nabla \times \vec H)=0$ in \eqref{eq6.4ii}, we find that $u$ satisfies
\[
-\Delta u=\nabla \cdot(\vec H\times(\rho(u)\nabla \times \vec H))\quad\text{in }\;\Omega.
\]
By \eqref{eq6.9nn}  and \eqref{eq6.10qv}, we see that $\vec \Phi:=\vec H\times(\rho(u)\nabla \times \vec H)$ satisfies the following estimate:
\begin{equation}					\label{eq6.11ub}
\int _{\Omega_r(x_0)}\abs{\vec \Phi}^2\leq N r^{1+2\alpha} \norm{\vec \Psi}_{W^{1,q}(\Omega)}^4,\quad \forall x_0\in\Omega,\;\; \forall r \in (0, \diam\Omega).
\end{equation}
Therefore,  $u$ is a solution of the Dirichlet problem
\[
\left\{\begin{array}{c}
-\Delta u = \nabla\cdot \vec \Phi\quad\text{in }\;\Omega,\\
u = \phi\quad\text{on }\;\partial\Omega,
\end{array}\right.
\]
where $\vec \Phi$ satisfies the Morrey-Campanato type estimate \eqref{eq6.11ub} and $\phi\in W^{1,q}(\Omega)$, and thus by a well-known elliptic regularity theory, we have
\[
\norm{u}_{C^\alpha(\overline\Omega)} \leq N \left( \norm{\vec \Psi}_{W^{1,q}(\Omega)}^2 + \norm{\phi}_{W^{1,q}(\Omega)}\right).
\]
In particular, we see that $\vec H$ and $u$ are bounded in $\Omega$.
The proof is complete.
\end{proof}

\begin{remark}					\label{rmk:2-1}
In Theorem~\ref{thm:2-1}, if one assumes further that $\rho\in C^k(\bR)$, where $k\in \bZ_{+}$, then by Theorem~\ref{thm:app1} and the bootstrapping method, one finds that $\vec H \in C^{k,\alpha}(\Omega)\cap C^\alpha(\overline\Omega)$ and $u \in C^{k+1,\alpha}(\Omega)\cap C^{\alpha}(\overline\Omega)$; see \cite[Theorem~3.2 and Remark~3.3]{KK02}.
In particular, if $\rho\in C^2(\bR)$, then the pair $(\vec H,u)$ becomes a classical solution of the problem \eqref{eq6.4ii}.
\end{remark}

\section{Green's function}					\label{sec:green}
In this section, we will discuss the Green's functions (more appropriately, it should be called Green's matrices) of the operator $L$ in arbitrary domains.
Let $\Sigma$ be any subset of $\overline{\Omega}$ and $u$ be a $Y^{1,2}(\Omega)$ function.
Then we shall say $u$ vanishes on $\Sigma$ (in the sense of $Y^{1,2}(\Omega)$) if $u$ is a limit in $Y^{1,2}(\Omega)$ of a sequence of functions in $C^\infty_0(\overline \Omega\setminus\Sigma)$.

\begin{definition}					\label{def2}
We say that a $3\times 3$ matrix valued function $\vec G(x,y)$, with entries $G_{ij}(x,y)$  defined on the set $\bigset{(x,y)\in\Omega\times\Omega: x\neq y}$, is a Green's function of $L$ in $\Omega$ if it satisfies the following properties:
\begin{enumerate}[i)]
\item
$\vec G(\cdot,y) \in W^{1,1}_{loc}(\Omega)$ and $L\vec G(\cdot,y)=\delta_y I$ for all $y\in\Omega$, in the sense that for $k=1,2, 3$,
\[
\int_{\Omega} a (\nabla\times \vec G(\cdot,y)\vec e_k) \cdot (\nabla \times \vec \phi) + b (\nabla \cdot \vec G(\cdot,y) \vec e_k)(\nabla\cdot \vec \phi)= \phi^k(y),\quad
\forall \vec \phi \in C^\infty_0(\Omega),
\]
where $\vec e_k$ denotes the $k$-th unit column vector; i.e., $\vec e_1=(1,0,0)^T$, etc.
\item
$\vec G(\cdot,y) \in Y^{1,2}(\Omega\setminus B_r(y))$ for all $y\in\Omega$ and $r>0$ and $\vec G(\cdot,y)$ vanishes on $\partial\Omega$.
\item
For any $\vec f \in C^\infty_0(\Omega)$, the function $\vec u$ given by
\[
\vec u(x):=\int_\Omega \vec G(y,x) \vec f(y)\,dy
\]
is a weak solution $Y^{1,2}_0(\Omega)$ of the problem \eqref{eq2.3ws}; i.e., $\vec u$ belongs to $Y^{1,2}_0(\Omega)$ and satisfies $L \vec u=\vec f$ in the sense of the identity \eqref{eq2.1ws}.
\end{enumerate}
\end{definition}
We note that part iii) of the above definition gives the uniqueness of a Green's matrix; see  Hofmann and Kim \cite{HK07}.
We shall hereafter say that $\vec G(x,y)$ is the Green's matrix of $L$ in $\Omega$ if it satisfies all the above properties.
Then, by using Theorem~\ref{thm3.2a} and following the proof of \cite[Theorem~4.1]{HK07}, we obtain the following theorem, where we use the notation
\[
a\wedge b:=\min(a,b),\quad a \vee b:=\max(a,b),\quad \text{where }\;a,b \in \bR.
\]

\begin{theorem}					\label{thm5.6gr}
Let $\Omega$ be a (possibly unbounded) domain in $\bR^3$.
Denote $d_x:=\dist(x,\partial\Omega)$ for $x\in\Omega$; we set $d_x=\infty$ if $\Omega=\bR^3$.
Then, there exists a unique Green's function $\vec G(x,y)$ of the operator $L$ in $\Omega$, and for all $x, y\in\Omega$ satisfying $0<\abs{x-y}<d_x \wedge d_y$, we have
\begin{equation}					\label{eq5.7gr}
\abs{\vec G(x,y)} \leq N \abs{x-y}^{-1},\quad \text{where }\;N=N(\nu)>0.
\end{equation}
Also, we have $\vec G(x,y)=\vec G(y,x)^T$ for all $x, y\in \Omega$ with $x\neq y$.
Moreover, $\vec G(\cdot,y)\in C^\alpha(\Omega\setminus\set{y})$ for some $\alpha=\alpha(\nu) \in(0,1)$ and satisfies the following estimate:
\begin{equation}					\label{eq5.8gr}
\abs{\vec G(x,y)-\vec G(x',y)} \leq N \abs{x-x'}^{\alpha} \abs{x-y}^{-1-\alpha},\quad \text{where }\;N=N(\nu)>0,
\end{equation}
provided that $\abs{x-x'}<\abs{x-y}/2\,$ and $\abs{x-y}<d_x\wedge d_y$.
\end{theorem}

Next, we consider the Green's functions of the system \eqref{eq0.1aa}.
\begin{definition}					\label{def3g}
We say that a $3\times 3$ matrix valued function $\vec G(x,y)$, which is defined on the set $\bigset{(x,y)\in\Omega\times\Omega: x\neq y}$, is a Green's function of the system \eqref{eq0.1aa} in $\Omega$ if it satisfies the following properties:
\begin{enumerate}[i)]
\item
$\vec G(\cdot,y) \in W^{1,1}_{loc}(\Omega)$ for all $y\in\Omega$ and for $k=1,2, 3$, we have
\begin{align*}
\int_{\Omega} a (\nabla\times \vec G(\cdot,y)\vec e_k)\cdot (\nabla \times \vec \phi) &= \phi^k(y),\quad
\forall \vec \phi \in C^\infty_0(\Omega),\\
\int_\Omega \vec G(\cdot,y)\vec e_k \cdot \nabla \psi &= 0, \quad \forall \psi \in C^\infty_0(\Omega),
\end{align*}
where $\vec e_k$ denotes the $k$-th unit column vector; i.e., $\vec e_1=(1,0,0)^T$, etc.
\item
$\vec G(\cdot,y) \in Y^{1,2}(\Omega\setminus B_r(y))$ for all $y\in\Omega$ and $r>0$ and $\vec G(\cdot,y)$ vanishes on $\partial\Omega$.
\item
For any $\vec f \in \mathcal D(\Omega)$, the function $\vec u$ given by
\[
\vec u(x):=\int_\Omega \vec G(y,x) \vec f(y)\,dy
\]
is a weak solution in $Y^{1,2}_0(\Omega)$ of the problem
\[
\left\{
\begin{array}{c}
\nabla\times (a(x)\nabla\times \vec u)=\vec f \quad\text{in }\;\Omega,\\
\nabla \cdot \vec u=0\quad\text{in }\;\Omega,\\
\vec u=0\quad \text{on }\;\partial\Omega,
\end{array}
\right.
\]
that is, $\vec u$ belongs to $Y^{1,2}_0(\Omega)$ and satisfies the above system in the sense of the identities \eqref{eq2.4ws} and \eqref{eq2.5ws} with $\vec g=0$ and $h=0$.
\end{enumerate}
\end{definition}

Then by the same reasoning as above, Theorem~\ref{thm5.6gr} also applies to the Green's functions of the system \eqref{eq0.1aa}.
Moreover, in the case when $\Omega$ is a bounded Lipschitz domain satisfying the condition \eqref{eq2.5ef}, a global version of estimate \eqref{eq5.7gr} is available thanks to Theorem~\ref{thm3.1t} and \cite[Theorem~3.13]{KK10}.

\begin{theorem}					\label{thm5.8gr}
The statement of Theorem~\ref{thm5.6gr} remains valid for the Green's functions of the system \eqref{eq0.1aa}.
Moreover, if we assume that $\Omega$ is a bounded Lipschitz domain satisfying the condition \eqref{eq2.5ef}, then for all $x, y\in\Omega$ with $x\neq y$, we have
\[
\abs{\vec G(x,y)}  \leq N \bigset{d_x\wedge \abs{x-y}}^{\alpha} \bigset{d_y\wedge \abs{x-y}}^{\alpha} \abs{x-y}^{-1-2\alpha},
\]
where $\alpha=\alpha(\nu,\Omega) \in (0,1)$ and $N=N(\nu, \Omega)$.
\end{theorem}

\begin{remark}					\label{rmk6.7gr}
Theorem~\ref{thm5.6gr} in particular establishes the existence of the Green's function of the operator $L$ in $\bR^3$, which is usually referred to as the fundamental solution of the operator $L$.
Notice that in that case, we have the pointwise estimate \eqref{eq5.7gr} available for all $x, y\in \bR^3$ with $x\neq y$, and estimate \eqref{eq5.8gr} for all $x, x'$ satisfying $\abs{x-x'}<\abs{x-y}/2$.
The various estimates for the Green's function that appears in \cite[Theorem~4.1]{HK07} are also available in Theorem~\ref{thm5.6gr}.
\end{remark}

\section{Associated parabolic system}				\label{sec:p}
In this separate and independent section, we consider the system of equations
\begin{equation}					\label{eq4.1ax}
\vec u_t+\nabla\times (a(x)\nabla\times \vec u) - \nabla (b(x) \nabla\cdot \vec u)=\vec f \quad\text{in }\;\Omega\times (0,T),
\end{equation}
and prove that weak solutions of the system \eqref{eq4.1ax} are H\"older continuous in $\Omega\times (0,T)$ provided that $\vec f$ satisfies some suitable condition, which is an extension of \cite[Theorem~3.1]{KKM}, where it is shown that weak solutions of the following system are H\"older continuous:
\begin{equation}					\label{eq7.2ps}
\vec u_t+\nabla\times (a(x)\nabla\times \vec u)=0,\quad \nabla \cdot \vec u=0 \quad\text{in }\;\Omega\times (0,T).
\end{equation}
As mentioned in the introduction, the above system arises naturally from Maxwell's equations in a quasi-static electromagnetic field.
More precisely,  let $\sigma(x)$ denote the electrical conductivity of a material and the vector $\vec H(x,t)$ represent the magnetic field.
It is shown in Landau et al. \cite[Ch.~VII]{LLP} that in the quasi-static electromagnetic fields, $\vec H$ satisfies the equations
\[
\vec H_t+\nabla \times \left(\tfrac{1}{\sigma}\nabla\times \vec H\right)=0,\quad \nabla \cdot\vec H=0\quad\text{in }\;\Omega\times(0,T),
\]
which is a special case of the system \eqref{eq4.1ax}.
Also, in this section we study the Green's functions of the system \eqref{eq4.1ax} and the system \eqref{eq7.2ps}, by using recent results from \cite{CDK, CDK10}.

\subsection{Notation and definitions}
In this section, we abandon some notations introduced in Section~\ref{sec:main}. 
Instead, we follow the notations of Ladyzhenskaya et al. \cite{LSU} with a slight variation.
We denote by $Q_T$ the cylindrical domain $\Omega\times (0,T)$, where $T>0$ is a fixed but arbitrary number, and $S_T$ the lateral surface of $Q_T$; i.e., $S_T=\partial\Omega\times [0,T]$.
Parabolic function spaces such as $L_{q,r}(Q_T)$, $L_q(Q_T)$, $W^{1,0}_2(Q_T)$, $W^{1,1}_2(Q_T)$, $V_2(Q_T)$, and $V^{1,0}_2(Q_T)$ are exactly those defined in Ladyzhenskaya et al. \cite{LSU}.
We define the parabolic distance between the points $X=(x,t)$ and $Y=(y,s)$ by
\[
\abs{X-Y}_p:=\max(\abs{x-y}, \sqrt{\abs{t-s}})
\]
and define the parabolic H\"older norm as follows:
\[
\abs{u}_{\alpha/2,\alpha;Q}=[u]_{\alpha/2, \alpha;Q}+\abs{u}_{0;Q}: = \sup_{\substack{X, Y \in Q\\ X\neq Y}} \frac{\abs{u(X)-u(Y)}}{\abs{X-Y}_p^ \alpha} + \sup_{X\in Q}\,\abs{u(X)}.
\]
We write $\nabla u$ for the spatial gradient of $u$ and $u_t$ for its time derivative.
We define
\[
Q^{-}_r(X)= B_r(x)\times (t-r^2,t),\quad
Q_r(X)=B_r(x)\times (t-r^2,t+r^2).
\]
We denote by $\cL$ the operator $\partial_t+L$; i.e.,
\[
\cL \vec u:=\vec u_t+ L \vec u= \vec u_t+\nabla\times (a(x)\nabla\times \vec u) - \nabla (b(x) \nabla\cdot \vec u),
\]
and by $\cLt$ the adjoint operator $-\partial_t+L$.
For a cylinder $Q$ of the form $\Omega\times (a,b)$, where $-\infty\leq a<b\leq \infty$, we say that $\vec u$ is a weak solution in $V_2(Q)$ ($V^{1,0}_2(Q)$) of $\cL \vec u =\vec f$ if $\vec u \in V_2(Q)$ ($V^{1,0}_2(Q)$) and satisfies the identity
\[
-\int_{Q} \vec u \cdot \vec \phi_t+ \int_{Q} a (\nabla \times \vec u) \cdot (\nabla\times \vec \phi) + b (\nabla \cdot \vec u)(\nabla \cdot \vec \phi)= \int_{Q} \vec f \cdot \vec \phi, \quad \forall \vec \phi \in C^\infty_0(Q).
\]
Similarly, we say that $\vec u$ is a weak solution in $V_2(Q)$ ($V^{1,0}_2(Q)$) of $\cLt \vec u =\vec f$ if $\vec u \in V_2(Q)$ ($V^{1,0}_2(Q)$) and satisfies the identity
\[
\int_{Q} \vec u \cdot \vec \phi_t+ \int_{Q} a (\nabla \times \vec u) \cdot (\nabla\times \vec \phi) + b (\nabla \cdot \vec u)(\nabla \cdot \vec \phi)= \int_{Q} \vec f \cdot \vec \phi, \quad \forall \vec \phi \in C^\infty_0(Q).
\]
\subsection{H\"older continuity estimates}
The following theorem is a parabolic analogue of Theorem~\ref{thm3.2a}.
However, it should be clearly understood that in the theorem below, the coefficients $a$ and $b$ of the system \eqref{eq4.1ax} are assumed to be time-independent.
\begin{theorem}			\label{thm4.2b}
Let $Q_T=\Omega\times(0,T)$, where $\Omega$ be a domain in $\bR^3$.
Assume that $a(x)$ and $b(x)$ are measurable functions on $\Omega$ satisfying \eqref{eq0.2bh}.
Let $\vec u$ be a weak solution in $V_2(Q_T)$ of the system \eqref{eq4.1ax} with $\vec f \in L_q(Q_T)$ for some $q>5/2$.
Then $\vec u$ is H\"older continuous in $Q_T$, and for any $Q^{-}_R=Q^{-}_R(X_0) \subset\subset Q_T$, we have the following estimate for $\vec u$ in $Q^{-}_{R/2}$:
\begin{equation}			\label{eq4.3hi}
R^\alpha [\vec u]_{\alpha, \alpha/2;Q^{-}_{R/2}} + \abs{\vec u}_{0;Q^{-}_{R/2}} \leq N \left( R^{-5/2} \norm{\vec u}_{L_2(Q^{-}_R)}+ R^{2-5/q} \norm{\vec f}_{L_q(Q^{-}_R)}\right),
\end{equation}
where $\alpha=\alpha(\nu, q) \in (0,1)$ and $N=N(\nu,q)>0$.
\end{theorem}

The proof of the above theorem will be given in \S \ref{sec7.3p} below.

\begin{remark}
As in \cite[Theorem~3.2]{KKM}, one can consider the case when the coefficients $a$ and $b$ of the system \eqref{eq4.1ax} are time-dependent but still have some regularity in $t$-variable.
For a measurable function $f=f(X)=f(x,t)$, we set
\[
\omega_\delta(f):=\sup_{X=(t,x)\in\bR^4} \sup_{r\le \delta} \frac{1}{\abs{Q_r(X)}} \int_{t-r^2}^{t+r^2} \!\int_{B_r(x)} \abs{f(y,s)-\bar f_{t,r}(y)}\,dy\,ds, \quad\forall \delta>0,
\]
where $\bar f_{t,r}(y)=\fint_{t-r^2}^{t+r^2} f(y,s)\,ds$.
We say that $f$ belongs to $\VMO_t$ if $\lim_{\delta\to 0} \omega_\delta(f)=0$.
Assume that the coefficients $a(x,t)$ and $b(x,t)$ are defined in the entire space $\bR^4$ and belong to $\VMO_t$.
Let $\vec u\in V_2(Q_T)$ be a weak solution of the system
\[
\vec u_t+\nabla\times (a(x,t)\nabla\times \vec u) - \nabla (b(x,t) \nabla\cdot \vec u)=\vec f \quad\text{in }\;Q_T,
\]
where $\vec f \in L_q(Q_T)$ with $q>5/2$. Then one can show that $\vec u$ is H\"older continuous in $Q_T$.
The proof is very similar to that of \cite[Theorem~3.2]{KKM}.
Also, as is mentioned in Remark~\ref{rmk2.5rr}, one may assume that $a$ is a $3\times 3$ (possibly non-symmetric) matrix valued function satisfying the uniform ellipticity and boundedness condition; see \cite{KKM} and also consult \cite{Kim} for treatment of non-symmetric coefficients.
\end{remark}

\begin{remark}					\label{rmk4.5ff}
In Theorem~\ref{thm4.2b}, instead of assuming that $\vec f\in L_q(Q_T)$, one may assume that $\vec f$ belongs to the mixed norm space $L_{q,r}(Q_T)$ with suitable $q$ and $r$.
In fact, one may assume that $\vec f$ belongs to the Morrey space, $M^{10/7,10(3+2\delta)/7}$ with $\delta\in(0,1)$, where $M^{p,q}$ is the set of all functions $f\in L_p(Q_T)$ with finite norm (c.f. Lieberman \cite[\S VI.7]{Lieberman})
\[
\norm{u}_{M^{p,q}}=\sup_{Q^{-}_r(X_0)\subset Q_T} \left(r^{-q}\int_{Q^{-}_r(X_0)} \abs{u}^p\,\right)^{1/p}.
\]
Then, instead of the estimate \eqref{eq7.33} in the proof of Theorem~\ref{thm4.2b}, we would have
\[
\int_{Q^{-}_r(X)}\abs{\nabla\vec w}^2 \leq N\norm{\vec f}^2_{L_{10/7}(Q^{-}_r(X))} \leq N r^{3+2\delta} \norm{\vec f}^2_{M^{10/7,10(3+2\delta)/7}}.
\]
The rest of proof remains essentially the same.
\end{remark}

\subsection{Green's function}
Let  $U=\Omega\times \bR$ be an infinite cylinder in with the base $\Omega$ being a (possibly unbounded) domain in $\bR^3$ and let $\partial U$ be its (parabolic) boundary $\partial\Omega\times\bR$.
Let $\cS \subset \overline Q$ and $u$ be a $W^{1,0}_2(Q)$ function.
We say that $u$ vanishes (or write $u=0$) on $\cS$ if $u$ is a limit in $W^{1,0}_2(Q)$ of a sequence of functions in $C^\infty_0(\overline Q\setminus \cS)$.

\begin{definition}
We say that a $3\times 3$ matrix valued function $\vec G(X,Y)=\vec G(x,t,y,s)$, with entries $G_{ij} (X,Y)$ defined on the set $\bigset{(X,Y)\in U\times U: X\neq Y}$, is a Green's function of the operator $\cL$ in $U$ if it satisfies the following properties:
\begin{enumerate}[i)]
\item
$\vec G(\cdot,Y)\in W^{1,0}_{1,loc}(U)$ and $\cL \vec G(\cdot,Y) = \delta_Y I$ for all $Y\in U$, in the sense that for $k=1,2,3$, the following identity holds for all $\vec \phi \in C^\infty_0(U)$:
\[
\int_{U} -\vec G(\cdot,Y) \vec e_k \cdot \vec \phi_t+ a (\nabla\times \vec G(\cdot,y)\vec e_k)\cdot (\nabla \times \vec \phi) + b (\nabla \cdot \vec G(\cdot,y) \vec e_k)(\nabla\cdot \vec \phi)= \phi^k(Y),
\]
where $\vec e_k$ denotes the $k$-th unit column vector; i.e., $\vec e_1=(1,0,0)^T$, etc.
\item
$\vec G(\cdot,Y) \in V_2^{1,0}(U\setminus Q_r(Y))$ for all $Y\in U$ and $r>0$ and $\vec G(\cdot,Y)$ vanishes on $\partial U$.
\item
For any $\vec f\in C^\infty_0(U)$, the function $\vec u$ given by
\[
\vec u(X):=\int_U \vec G(Y,X) \vec f(Y)\,dY
\]
is a weak solution in $V^{1,0}_2(U)$ of $\cLt \vec u=\vec f$ and vanishes on $\partial U$.
\end{enumerate}
\end{definition}
We note that part iii) of the above definition gives the uniqueness of a Green's function; see \cite{CDK}.
We shall thus say that $\vec G(X,Y)$ is the Green's function of $\cL$ in $U$ if it satisfies the above properties.
By Theorem~\ref{thm4.2b} and \cite[Theorem~2.7]{CDK}, we have the following theorem:

\begin{theorem}					\label{thm1hk}
Let $ U=\Omega\times \bR$ be an infinite cylinder, where the base $\Omega$ is a (possibly unbounded) domain in $\bR^3$.
Then the Green's function $\vec G(X,Y)$ of $\cL$ exists in $U$ and satisfies 
\begin{equation}					\label{eq7.9cc}
\vec G(x,t,y,s)=\vec G(x,t-s,y,0);\quad \vec G(x,t,y,0)\equiv 0\;\;\text{for}\; t<0.
\end{equation}
For all $\vec f\in C^\infty_0(U)$, the function $\vec u$ given by
\begin{equation}					\label{eqn:E-70}
\vec u(X):=\int_{U} \vec G(X,Y)\vec f(Y)\,dY
\end{equation}
is a weak solution in $V^{1,0}_2(U)$ of $\cL\vec u=\vec f$ and vanishes on $\partial U$.
Moreover, for all $\vec g\in L^2(\Omega)$, the function $\vec u(x,t)$ defined by
\[
\vec u(x,t):=\int_{\Omega} \vec G(x,t,y,0)\vec g(y)\,dy
\]
is a unique weak solution in $V^{1,0}_2(Q_T)$ of the problem\footnote{See, Ladyzhenskaya et al. \cite[\S III.1]{LSU}}
\[
\cL \vec u =0,\quad \vec u \big|_{S_T}=0,\quad \vec u \big|_{t=0}=\vec g(x),
\]
and if $\vec g$ is continuous at $x_0\in\Omega$ in addition, then we have
\[
\lim_{\substack{(x,t)\to (x_0,0)\\ x\in\Omega,\,t>0}} \vec u(x,t)=\vec g(x_0).
\]
\end{theorem}

\begin{remark}					\label{rmk7.11hk}
The identity $\vec G(x,t,y,s)=\vec G(x,t-s,y,0)$ in Theorem~\ref{thm1hk} comes from the fact that $\cL$ has time-independent coefficients; see \cite{DK09}.
The function $\vec K_t(x,y)$ defined by
\begin{equation}					\label{eq7.12qm}
\vec K_t(x,y)=\vec G(x,t,y,0),\quad x,y\in\Omega,\;\; t>0
\end{equation}
is usually called \textit{the (Dirichlet) heat kernel} of the elliptic operator $L$ in $\Omega$.
It is known that $\vec K_t$ satisfies the semi-group property
\[
\vec K_{t+s}(x,y)=\int_\Omega \vec K_t(x,z) \vec K_s(z,y)\,dz, \quad\forall x,y\in\Omega,\;\;\forall t,s>0,
\]
and in particular, if $\Omega=\bR^3$, then we also have the following identity:
\[
\int_{\bR^3} \vec K_t(x,y)\,dy=I, \quad\forall x\in\bR^3,\;\;\forall t>0,
\]
where $I$ denotes the $3\times 3$ identity matrix; see \cite[Theorem~2.11 and Remark~2.12]{CDK}.
\end{remark}

The following theorem is another consequence of Theorem~\ref{thm4.2b}; see \cite[Theorem~2.11]{CDK}.

\begin{theorem}					\label{thm2hk}
Let $\vec K_t(x,y)$ be the heat kernel for the operator $L$ in $\bR^3$ as constructed in Theorem~\ref{thm1hk}.
Then we have the following Gaussian bound for the heat kernel:
\[
\abs{\vec K_t(x,y)} \leq N t^{-3/2}\exp\{-\kappa|x-y|^2/ t \},\quad \forall t>0,\;\; x,y\in\bR^3,
\]
where $N=N(\nu)>0$ and $\kappa=\kappa(\nu)>0$.
\end{theorem}

Next, we consider the Green's functions of the system \eqref{eq7.2ps}.
\begin{definition}					\label{def3ks}
We say that a $3\times 3$ matrix valued function $\vec G(X,Y)=\vec G(x,t,y,s)$, with entries $G_{ij} (X,Y)$ defined on the set $\bigset{(X,Y)\in U\times U: X\neq Y}$, is a Green's function of the system \eqref{eq7.2ps} in $U$ if it satisfies the following properties:
\begin{enumerate}[i)]
\item
$\vec G(\cdot,Y)\in W^{1,0}_{1,loc}(U)$ for all $Y\in U$ and for $k=1,2, 3$, we have
\begin{align*}
\int_{U} -\vec G(\cdot,Y) \vec e_k \cdot \vec \phi_t+ a (\nabla\times \vec G(\cdot,y)\vec e_k)\cdot (\nabla \times \vec \phi) &= \phi^k(Y), \quad \forall \vec \phi \in C^\infty_0(U),\\
\int_U \vec G(\cdot,Y)\vec e_k \cdot \nabla \psi &= 0, \quad \forall \psi \in C^\infty_0(U),
\end{align*}
where $\vec e_k$ denotes the $k$-th unit column vector; i.e., $\vec e_1=(1,0,0)^T$, etc.
\item
$\vec G(\cdot,Y) \in V_2^{1,0}(U\setminus Q_r(Y))$ for all $Y\in U$ and $r>0$ and $\vec G(\cdot,Y)$ vanishes on $\partial U$.
\item
For any $\vec f\in C^\infty_0(U)$ satisfying $\nabla\cdot \vec f =0$ in $U$, the function $\vec u$ defined by
\[
\vec u(X):=\int_\Omega \vec G(Y,X) \vec f(Y)\,dY
\]
is a weak solution in $V^{1,0}_2(U)$ of the problem
\[
-\vec u_t+\nabla\times (a(x)\nabla\times \vec u)=\vec f,\quad \nabla \cdot \vec u=0,\quad \vec u\big|_{\partial U}=0,
\]
that is, $\vec u$ belongs to $V^{1,0}_2(U)$, vanishes on $\partial U$, and satisfies the above system in the sense of the following identities:
\begin{align*}
\int_{U} \vec u \cdot \vec \phi_t+ a (\nabla \times \vec u) \cdot (\nabla\times \vec \phi)& = \int_{U} \vec f \cdot \vec \phi, \quad \forall \vec \phi \in C^\infty_0(U).\\
\int_\Omega \vec u \cdot \nabla \psi &=0,\quad \forall \psi\in C^\infty_0(U).
\end{align*}
\end{enumerate}
\end{definition}

It can be easily seen that existence of the Green's function of the system \eqref{eq7.2ps} in $U$ follows from  \cite[Theorem~3.1]{KKM} and \cite[Theorem~2.7]{CDK}, and that it satisfies the relations \eqref{eq7.9cc} in Theorem~\ref{thm1hk}.
We shall say that $\vec K_t$ defined by the formula \eqref{eq7.12qm} is \emph{the (Dirichlet) heat kernel} of the elliptic system \eqref{eq0.1aa} in $\Omega$.
Then it satisfies the statement in Remark~\ref{rmk7.11hk} as well as that in Theorem~\ref{thm2hk}.
If we assume further that $\Omega$ is a domain satisfying the hypothesis of Theorem~3.5, then we have the following result, which is an easy consequence of \cite[Theorem~3.6]{CDK10} combined with Theorem~\ref{thm3.1t} and \cite[Lemma~4.4]{DK09} (see also \cite [Remark~3.10]{CDK10}):

\begin{theorem}			\label{thm3hk}
Let $ U=\Omega\times \bR$ with $\Omega$ satisfying the hypothesis of Theorem~\ref{thm3.1t}.
Then the heat kernel $\vec K_t(x,y)$ of the system \eqref{eq0.1aa} exists in $\Omega$.
Moreover, for all $T>0$ there exists a constant $N=N(\nu,\Omega,T)$ such that for all $x,y \in \Omega$ and $0<t \leq T$, we have
\[
\abs{\vec K_t(x,y)} \leq N \left(1 \wedge \frac {d_x} {\sqrt {t} \vee \abs{x-y}} \right)^{\alpha} \left(1 \wedge \frac{d_y} {\sqrt {t} \vee \abs{x-y}}\right)^{\alpha}\,t^{-3/2}\exp \set{-\kappa \abs{x-y}^2/t},
\]
where $\kappa=\kappa(\nu,\Omega)>0$ and $\alpha=\alpha(\nu,\Omega) \in (0,1)$ are constants independent of $T$, and we used the notation $a\wedge b=\min(a,b)$, $a\vee b=\max(a,b)$, and $d_x=\dist(x,\partial\Omega)$.
\end{theorem}

\subsection{Proof of Theorem~\ref{thm4.2b}}				\label{sec7.3p}
We follow the strategy used in \cite{KKM}.
As before, we shall make the qualitative assumption that the weak solution $\vec u$ is smooth in $Q_T$.
Let us first assume that $\vec f =0$ and consider the homogeneous system
\begin{equation}					\label{eq4.4dp}
\vec u_t+L \vec u :=\vec u_t+\nabla\times (a(x)\nabla\times \vec u) - \nabla (b(x) \nabla\cdot \vec u)=0 \quad\text{in }\; Q_T.
\end{equation}
The proof of the following lemma is very similar to that of \cite[Lemma~3.1 -- 3.3]{KKM}, where we strongly used the assumption that coefficients of the operator are time-independent.
\begin{lemma}					\label{lem4.5}
Let $\vec v\in V_2(Q^{-}_{\lambda r})$, where $Q^{-}_{\lambda r}=Q^{-}_{\lambda r}(X_0)$ with $\lambda>1$, be a weak solution of $\vec v_t+L\vec v =0$ in $Q^{-}_{\lambda r}$.
Then we have the following estimates:
\begin{align*}
\sup_{t_0-r^2\leq t \leq t_0}\int_{B_r}\abs{\vec v(\cdot,t)}^2 + \int_{Q^{-}_r} \abs{\nabla \vec v}^2 &\leq N r^{-2} \int_{Q^{-}_{\lambda r}}\abs{\vec v}^2,\\
\sup_{t_0-r^2\leq t \leq t_0}\int_{B_r}\abs{\nabla \vec v(\cdot,t)}^2 + \int_{Q^{-}_r} \abs{\vec v_t}^2 & \leq N r^{-4} \int_{Q^{-}_{\lambda r}}\abs{\vec v}^2,\\
\sup_{t_0-r^2\leq t \leq t_0}\int_{B_r}\abs{\vec v_t(\cdot,t)}^2 + \int_{Q^{-}_r} \abs{\nabla \vec v_t}^2 & \leq N r^{-6} \int_{Q^{-}_{\lambda r}}\abs{\vec v}^2.
\end{align*}
where $N=N(\nu, \lambda)>0$.
\end{lemma}

The proofs of following lemmas are also standard in parabolic theory and shall be omitted;  see e.g., \cite[Lemma~2.4 and~3.1]{CDK} and also \cite[Lemma~8.6]{CDK10}.
\begin{lemma}					\label{lem4.6}
Let $\vec u\in V_2(Q^{-}_r)$, where $Q^{-}_r=Q^{-}_r(X_0)$, be a weak solution of $\vec u_t+L \vec u= \vec f$ in $Q^{-}_r$.
Then we have the estimate
\[
\int_{Q^{-}_r} \abs{\vec u - \vec u_{X_0,r}}^2  \leq N \left(r^2 \int_{Q^{-}_r} \abs{\nabla \vec u}^2 + r^{-1}\norm{\vec f}_{L_1(Q^{-}_r)}^2\right);\quad \vec u_{X_0,r} = \fint_{Q^{-}_r(X_0)} \vec u.
\]
where $N=N(\nu)>0$.
\end{lemma}
\begin{lemma}			\label{lem4.7}
Let $\vec u\in V_2(Q^{-}_{\lambda r})$, where $Q^{-}_{\lambda r}=Q^{-}_{\lambda r}(X_0)$ with $\lambda>1$, be a weak solution of $\vec u_t+L\vec u =\vec f$ in $Q^{-}_{\lambda r}$.
Then we have
\[
\sup_{t_0-r^2\leq t \leq t_0}\int_{B_r}\abs{\vec u(\cdot,t)}^2 + \int_{Q^{-}_r} \abs{\nabla \vec u}^2 \leq N \left(r^{-2} \int_{Q^{-}_{\lambda r}} \abs{\vec u}^2+ \norm{\vec f }_{L_{10/7}(Q^{-}_{\lambda r})}^2\right),
\]
where $N=N(\nu,\lambda)>0$.
\end{lemma}

With the above lemmas and Theorem~\ref{thm3.2a} at hand, we now proceed as in the proof of \cite[Theorem~3.1]{KKM} (see also proof of \cite[Theorem~3.3]{Kim}) to conclude that any weak solution $\vec v \in V_2(Q_T)$ of the system \eqref{eq4.4dp} is H\"older continuous in $Q_T$ and satisfies the estimate
\begin{equation}					\label{eq4.10mb}
[\vec v]_{\mu,\mu/2;Q^{-}_{R/2}} \leq N R^{-5/2-\mu} \norm{\vec v}_{L_2(Q^{-}_R)};\quad Q^{-}_R=Q^{-}_R(X_0),
\end{equation}
where $\mu=\mu(\nu) \in (0,1)$ and $N=N(\nu)>0$.
There is a well-known procedure to obtain H\"older estimates for weak solutions of the inhomogeneous system $\vec u_t+L\vec u = \vec f$ from the above estimate \eqref{eq4.10mb} for weak solutions of the corresponding homogeneous system $\vec u_t+L\vec u = 0$, which we shall demonstrate below for the completeness.
For $X\in Q^{-}_{R/4}(X_0)$ and $r\in (0,R/4]$, we split $\vec u=\vec v + \vec w$ in $Q^{-}_r(X)$, where $\vec w$ is the unique weak solution in $V^{1,0}_2(Q^{-}_r(X))$ of
$\vec w_t+L \vec w=\vec f$ in $Q^{-}_r(X)$ with zero boundary condition on the parabolic boundary $\partial_p Q^{-}_r(X)$.
Then, $\vec v=\vec u-\vec w$ satisfies
$\vec v_t+L \vec v=0$ in $Q^{-}_r(X)$, and thus, for $0<\rho \leq r$ (c.f. \cite[Eq.~(3.9)]{CDK}), we have
\begin{align}					\label{eq7.32}
\int_{Q^{-}_\rho(X)}\abs{\nabla \vec u}^2 &\leq  2\int_{Q^{-}_\rho(X)}\abs{\nabla \vec v}^2+2 \int_{Q^{-}_\rho(X)}\abs{\nabla \vec w}^2\\
\nonumber
&\leq  N(\rho/r)^{3+2\mu}\int_{Q^{-}_r(X)}\abs{\nabla \vec v}^2+2 \int_{Q^{-}_r(X)}\abs{\nabla \vec w}^2\\
\nonumber
&\leq  N(\rho/r)^{3+2\mu}\int_{Q^{-}_r(X)}\abs{\nabla \vec u}^2+N \int_{Q^{-}_r(X)}\abs{\nabla \vec w}^2.
\end{align}
Choose $p\in (5/2,q)$ such that $\alpha:=2-5/p<\mu$.
By the energy inequality and a parabolic embedding theorem (see \cite[\S II.3]{LSU}), we get (c.f. \cite[Eq.~(3.10)]{CDK})
\begin{equation}					\label{eq7.33}
\int_{Q^{-}_r(X)}\abs{\nabla\vec w}^2 \leq N\norm{\vec f}^2_{L_{10/7}(Q^{-}_r(X))} \leq N r^{3+2\alpha} \norm{\vec f}^2_{L_{p}(Q^{-}_{R/2})}.
\end{equation}
Combining \eqref{eq7.32} with \eqref{eq7.33}, we get for all $\rho<r \leq R/4$,
\[
\int_{Q^{-}_\rho(X)}\abs{\nabla \vec u}^2\leq N(\rho/r)^{3+2\mu}\int_{Q^{-}_r(X)}\abs{\nabla \vec u}^2+ Nr^{3+2\alpha}\norm{\vec f}^2_{L_p(Q^{-}_{R/2})}.
\end{equation*}
Then, by a well known iteration argument (see e.g., \cite[Lemma 2.1, p. 86]{Gi83}), we have
\[
\int_{Q^{-}_r(X)}\abs{\nabla \vec u}^2 \leq N(r/R)^{3+2\alpha}\int_{Q^{-}_{R/4}(X)}\abs{\nabla \vec u}^2+ N r^{3+2\alpha}\norm{\vec f}^2_{L_p(Q^{-}_{R/2})}.
\]
By Lemma~\ref{lem4.6}, the above estimate, and H\"older's inequality, we get
\[
\int_{Q^{-}_r(X)}\abs{\vec u- \vec u_{X,r}}^2\leq N r^{5+2\alpha}\left(R^{-3-2\alpha}\norm{\nabla \vec u}^2_{L_2(Q^{-}_{R/4}(X))}+\norm{\vec f}^2_{L_p(Q^{-}_{R/2})}\right).
\]
Then, by Campanato's characterization of H\"older continuous functions, we have
\[
[\vec u]_{\alpha,\alpha/2; Q^{-}_{R/4}} \leq N\left(R^{-3/2-\alpha}\norm{\nabla \vec u}_{L_2(Q^{-}_{R/2})}+\norm{\vec f}_{L_{p}(Q^{-}_{R/2})}\right).
\]
By Lemma~\ref{lem4.7} and H\"older's inequality (recall $\alpha=2-5/p$), we then obtain
\begin{equation} \label{eq10.28}
R^\alpha [\vec u]_{\alpha,\alpha/2; Q^{-}_{R/4}} \leq N\left(R^{-5/2}\norm{\vec u}_{L_2(Q^{-}_R)}+R^{2-5/q} \norm{\vec f}_{L_q(Q^{-}_{R})}\right).
\end{equation}
Similar to \eqref{eq4.41yx}, we then also obtain
\begin{equation}					\label{eq07zz}
\abs{\vec u}_{0; Q^{-}_{R/8}} \leq N\left(R^{-5/2}\norm{\vec u}_{L_2(Q^{-}_R)}+R^{2-5/q} \norm{\vec f}_{L_q(Q^{-}_{R})}\right).
\end{equation}
Finally, the desired estimate \eqref{eq4.3hi} follows from \eqref{eq10.28}, \eqref{eq07zz}, and the standard covering argument.
The theorem is proved.
\hfill\qedsymbol


\section{Appendix}					\label{sec:appendix}

\subsection{Existence of a unique weak solution of the problem \eqref{eq2.3ws}}
We prove existence of a unique weak solution in $Y^{1,2}_0(\Omega)$ of a more general problem
\begin{equation}					\label{eq9.1ax}
\left\{
\begin{array}{c}
\nabla\times (a(x)\nabla\times \vec u)-\nabla(b(x)\nabla \cdot \vec u)= \vec f  + \nabla\times \vec F + \nabla g\quad\text{in }\;\Omega,\\
\vec u=0\quad \text{on }\;\partial\Omega,
\end{array}
\right.
\end{equation}
where $\vec f \in L^{6/5}(\Omega)$ and $\vec F, g \in L^2(\Omega)$.
We say that  a function $\vec u$ is a weak solution in $Y^{1,2}_0(\Omega)$ of the problem \eqref{eq9.1ax} if $\vec u$ that belongs to $Y^{1,2}_0(\Omega)$ and satisfies the identity 
\[
\int_\Omega a (\nabla \times \vec u) \cdot (\nabla \times \vec v) + b (\nabla \cdot \vec u) (\nabla \cdot \vec v) = \int_\Omega \vec f \cdot \vec v +\vec F \cdot \nabla \times \vec v +  g \nabla \cdot \vec v, \quad \forall \vec v \in C^\infty_0(\Omega).
\]
Notice that the inequality \eqref{eqP-14} implies that the bilinear form
\begin{equation}					\label{eq9.2ax}
\ip{\vec u,\vec v}=\ip{\vec u,\vec v}_H= \sum_{i=1}^3 \int_\Omega \nabla u^i\cdot \nabla v^i
\end{equation}
defines an inner product on $H:=Y^{1,2}_0(\Omega)^3$ and that $H$ equipped with the above inner product  is a Hilbert space.
We define the bilinear form associated to the operator $L$ as
\[
B[\vec u,\vec v]:=\int_{\Omega} a (\nabla \times \vec u)\cdot (\nabla \times \vec v)+ b (\nabla \cdot \vec u)(\nabla \cdot \vec v).
\]
Then, in light of the identity \eqref{eq0.2ab}, we find that
\[
\int_\Omega \abs{\nabla \times \vec u}^2 + \abs{\nabla \cdot \vec u}^2 = \int_\Omega \abs{\nabla \vec u}^2,\quad \forall \vec u \in H.
\]
It is routine to check that the bilinear form $B$ satisfies the hypothesis of the Lax-Milgram Theorem.
On the other hand, by the inequality \eqref{eqP-14}, the linear functional
\[
F(\vec v):= \int_\Omega \vec f \cdot \vec v + \vec F \cdot \nabla\times \vec v + g \nabla \cdot \vec v
\]
is bounded on $H$.
Therefore, by the Lax-Milgram Theorem, there exists a unique element $\vec u\in H$ such that $B[\vec u, \vec v]=F(\vec v)$ for all $\vec v \in H$, which shows that $\vec u$ is a unique weak solution in $Y^{1,2}_0(\Omega)$ of the problem \eqref{eq9.1ax}.
\hfill\qedsymbol

\subsection{Existence of a unique weak solution of the problem \eqref{eq0.3cc}}
We shall assume that $\vec f \in H_{6/5}(\Omega)$ and $\vec g \in L^2(\Omega)$.
First, we consider the case when $h=0$ and construct a weak solution in $Y^{1,2}_0(\Omega)$ of the problem \eqref{eq0.3cc} as follows.
Let $H$ be the completion of $\mathcal D(\Omega)$ (see Section~\ref{sec:nd} for its definition) in the norm of $Y^{1,2}(\Omega)$.
Then $H \subset Y^{1,2}_0(\Omega)^3$ and as above, it equipped with the inner product \eqref{eq9.2ax} becomes a Hilbert space.
We define the bilinear form $B$ on $H$ by
\[
B[\vec u,\vec v]:=\int_\Omega a(\nabla \times \vec u)\cdot (\nabla \times \vec v).
\]
Then the bilinear form $B$ satisfies the hypothesis of the Lax-Milgram Theorem.
We also define the linear functional $F$ on $H$ as
\[
F(\vec v):= \int_\Omega \vec f \cdot \vec v + \vec g \cdot (\nabla \times \vec v).
\]
One can easily check that $F$ is bounded on $H$.
Therefore, by the Lax-Milgram Theorem, there exists a unique element $\vec u\in H$ such that $B[\vec u, \vec v]=F(\vec v)$ for all $\vec v \in H$.
In particular, $\vec u$ satisfies  identities \eqref{eq2.4ws} and \eqref{eq2.5ws} with $h=0$.
Therefore, $\vec u$ is a weak solution in $Y^{1,2}_0(\Omega)$ of the problem \eqref{eq0.3cc} in the case when $h=0$.

Next, we consider the case when $h\neq 0$.
In this case, we assume further that $\Omega$ is a bounded Lipschitz domain so that in particular, we have $Y^{1,2}_0(\Omega)=W^{1,2}_0(\Omega)$.
For $h \in L^2(\Omega)$ such that $\int_\Omega h=0$, let $\vec v \in W^{1,2}_0(\Omega)$ be a solution of the divergence problem
\[
\left\{
\begin{array}{c}
\nabla \cdot \vec v = h  \quad \text{in }\Omega\\
\vec v= 0 \quad \text{on }\;\partial \Omega,
\end{array}
\right.
\]
that satisfies the following estimate (see e.g., Galdi \cite[\S III.3]{Galdi})
\[
\norm{\nabla \vec v}_{L^2(\Omega)} \leq N \norm{h}_{L^2(\Omega)};\quad N=N(\Omega).
\]
Let $\vec w$ be a solution in $Y^{1,2}_0(\Omega)$ of the problem \eqref{eq0.3cc} with $\vec g-a \nabla\times \vec v$ in place of $\vec g$ and  $h=0$, which can be constructed as above. 
Then, it is easy to check that $\vec u:=\vec v + \vec w$ is a solution in $Y^{1,2}_0(\Omega)=W^{1,2}_0(\Omega)$ of the original problem \eqref{eq0.3cc}.

Finally, we prove the uniqueness of weak solutions in $Y^{1,2}_0(\Omega)$ of the problem \eqref{eq0.3cc} under the assumption that $\Omega$ is a bounded Lipschitz domain.
Notice that in that case we have $Y^{1,2}_0(\Omega)=W^{1,2}_0(\Omega)$.
Suppose $\vec u$ and $\vec v$ are two weak solutions in $W^{1,2}_0(\Omega)$ of the problem \eqref{eq0.3cc}.
Then the difference $\vec w=\vec u -\vec v$ is a weak solution in $W^{1,2}_0(\Omega)$ of the problem \eqref{eq0.3cc} with $\vec f=\vec g=0$ and $h=0$.
By the identity \eqref{eq2.5ws}, we find that $\vec w \in H$; see e.g., Galdi \cite[\S III.4]{Galdi}.
Then by the identity \eqref{eq2.4ws}, we conclude that $\vec w=0$, which proves the uniqueness of weak solutions in $Y^{1,2}_0(\Omega)$ of the problem \eqref{eq0.3cc}.
\hfill\qedsymbol

\begin{acknowledgment}
This work was supported by WCU(World Class University) program through the National Research Foundation of Korea(NRF) funded by the Ministry of Education, Science and Technology  (R31-2008-000-10049-0).
Kyungkeun Kang was supported by the  Korean Research Foundation Grant (MOEHRD, Basic Research Promotion Fund, KRF-2008-331-C00024) and the National Research Foundation of
Korea(NRF) funded by the Ministry of Education, Science and Technology (2009-0088692).
Seick Kim was supported by Basic Science Research Program through the National Research Foundation of Korea(NRF) funded by the Ministry of Education, Science and Technology (2010-0008224).
\end{acknowledgment}


\end{document}